\documentclass[11pt, a4paper]{article}

\usepackage[dvipsnames]{xcolor}
\usepackage{hyperref}

\usepackage[margin=2.5cm]{geometry}
\usepackage{mathtools}
\usepackage{amssymb}
\usepackage{amsthm}
\usepackage{amsmath}
\usepackage{mathrsfs}
\usepackage{empheq}

\usepackage{array,multirow,blkarray}

\usepackage[toc,page]{appendix}

\usepackage[nodayofweek]{datetime}

\usepackage{textcomp}

\usepackage{thmtools} 
\usepackage{thm-restate}

\usepackage{bbm}
\usepackage[]{dsfont}
\usepackage{authblk}
\usepackage{setspace}
\usepackage[shortlabels]{enumitem}
\usepackage[
  disable,
colorinlistoftodos, shadow, linecolor=white, backgroundcolor=white]{todonotes}

\usepackage{upgreek}

\newcommand{\Le}{\kappa}
\newcommand{\circuits}{\mathcal{C}}

\usepackage[linesnumbered,commentsnumbered,ruled,vlined]{algorithm2e}

\SetKw{run}{run}

\usepackage{bm}
\newcommand{\norm}[1]{\left\lVert#1\right\rVert}

\newcommand{\Diag}{\operatorname{diag}}
\newcommand{\pr}[2]{\left\langle #1, #2 \right\rangle}
\newcommand{\supp}{\mathrm{supp}}
\newcommand{\1}{\mathbbm{1}}
\newcommand{\rk}{\operatorname{rk}}
\newcommand{\kd}{{\dot{\kappa}}}
\newcommand{\primes}{\mathbb{P}}
\newcommand{\Q}{\mathbb{Q}}
\newcommand{\set}[1]{\left\{ #1 \right\}}
\newcommand{\divides}{\mid}
\newcommand{\st}{\mathrm{s.t.}}
\newcommand{\OPT}{\mathrm{OPT}}

\newcommand{\spn}{\operatorname{span}}

\newcommand{\KK}{\mathcal K}
\newcommand{\EE}{\mathcal F}

\newcommand{\bD}{\mathbf{D}}

\newcommand{\lcm}{\mathrm{lcm}}

\newcommand{\eps}{\varepsilon}
\newcommand{\proj}{\Pi}
\newcommand{\lal}{\mathrm{ll}}

\newcommand{\poly}{\operatorname{poly}}

\usepackage{accents}

\newcounter{oraclecf}
\newcounter{algDe orithm saved}

 

\usepackage{cleveref}

\usepackage[noadjust]{cite}
\usepackage{booktabs}

\usepackage{algorithmic}

\newcommand\myshade{100}
\colorlet{mylinkcolor}{NavyBlue}
\colorlet{mycitecolor}{YellowOrange}
\colorlet{myurlcolor}{Aquamarine}

\hypersetup{
  linkcolor  = mylinkcolor!\myshade!black,
  citecolor  = mycitecolor!\myshade!black,
  urlcolor   = myurlcolor!\myshade!black,
  colorlinks = true,
}

\newtheorem{theorem}{Theorem}[section]

\newtheorem{Def}[theorem]{Definition}
\newtheorem{lemma}[theorem]{Lemma}

\newtheorem{prop}[theorem]{Proposition}

\newtheorem{cor}[theorem]{Corollary}

\newtheorem{claim}{Claim}[theorem]
\newtheorem{conj}{Conjecture}[theorem]

\newtheoremstyle{case}{}{}{\itshape}{1em}{}{:}{ }{}
\theoremstyle{case}

\definecolor{darkred}{RGB}{180, 0, 0}
\definecolor{darkgreen}{rgb}{0, 0.6, 0}
\definecolor{darkblue}{RGB}{51,51,178}
\definecolor{lightgray}{RGB}{231,231,231}
\definecolor{lightblue}{RGB}{180,180,254}
\definecolor{lightred}{HTML}{FEB4B4}
\definecolor{darkcyan}{HTML}{7FBFBF}

\DeclareMathOperator*{\argmin}{arg\,min}

\newcounter{myquestion}

\newcounter{mycomment}
\newcommand{\Comment}[2][BN]{\refstepcounter{mycomment}{\setstretch{0.7}\todo[inline, 
backgroundcolor=black!30!white
, size=\small]{\textbf{Comment[#1\themycomment]:} \\ #2}}}

\renewcommand{\phi}{\varphi}
\renewcommand{\rho}{\varrho}

\newcommand{\R}{\mathbb{R}}
\newcommand{\Z}{\mathbb{Z}}
\newcommand{\N}{\mathbb{N}}

\newcounter{Hequation}

\makeatletter
\g@addto@macro\equation{\stepcounter{Hequation}}
\makeatother

\title{Circuit imbalance measures and linear programming
\thanks{This project has received funding from the European Research Council (ERC) under the European Union's Horizon 2020 research and innovation programme (grant agreement ScaleOpt--757481)}
}

\author{Farbod Ekbatani}
\author{Bento Natura} 
\author{L{\'{a}}szl{\'{o}} A. V{\'{e}}gh}
\affil{London School of Economics and Political Science \protect\\
 \texttt{\{F.Ekbatani, B.Natura, L.Vegh\}@lse.ac.uk}}

\begin{document}

\maketitle

\begin{abstract}
We study properties and applications of various circuit imbalance measures associated with linear spaces.
These measures describe possible ratios between nonzero entries of 
support-minimal nonzero vectors of  the space. The fractional circuit imbalance measure turns out to be a crucial parameter in the context of linear programming, and two integer variants can be used to describe integrality properties of associated polyhedra. 

We give an overview of the properties of these measures, and survey classical and recent applications, in particular, for linear programming algorithms with running time dependence on the constraint matrix only, and for circuit augmentation algorithms. We also present new bounds on the diameter and circuit diameter of polyhedra in terms of the fractional circuit imbalance measure.
\end{abstract}

\tableofcontents

\newpage

\listoftodos

\section{Introduction}
\label{sec:intro}

For a linear space $W\subset \R^n$, $g\in W$  is an 
 \emph{elementary vector} if $g$ is a support minimal nonzero vector in $W$, that is, no $h\in W\setminus\{0\}$ exists such that $\supp(h)\subsetneq \supp(g)$, where $\supp$ denotes the support of a vector. 
 A \emph{circuit} in $W$ is the support of some elementary vector; these are precisely the circuits in the associated  linear matroid $\mathcal{M}(W)$. We let $\EE(W)\subseteq W$ and $\circuits_W\subseteq 2^n$
 denote the set of elementary vectors and circuits in the space $W$, respectively.

Elementary vectors were first studied in the 1960s by Camion \cite{camion1964}, Tutte \cite{Tutte1965}, Fulkerson \cite{Fulkerson1968}, and Rockafellar \cite{RockafellarTheEV}.
Circuits play a crucial role in matroid theory and have been extremely well studied.  
For regular subspaces (i.e., kernels of totally unimodular matrices), elementary vectors have $\pm 1$ entries; this fact has been at the heart of several arguments in network optimization since the 1950s. 

The focus of this paper is on various  \emph{circuit imbalance measures}. We give an overview of classical and recent applications, and their relationship with other condition measures. We will mainly focus on applications in linear programming, mentioning in passing also their relevance to integer programming. 

\paragraph{Three circuit imbalance measures}
There are multiple ways to quantify  how `imbalanced' elementary vectors of a subspace can be. We define three different measures that capture various fractionality and integrality properties.

We will need some simple definitions.
The linear spaces $\{0\}$ and  $\R^n$ will be called \emph{trivial} subspaces; all other subspaces are \emph{nontrivial}. A linear subspace of $\R^n$ is a \emph{rational linear space} if it admits a basis of rational vectors. Equivalently, a rational linear space can be represented as the image of a rational matrix.
For an integer vector $v\in\Z^n$, let $\lcm(v)$ denote the least common multiple of the entries $|v_i|$, $i\in [n]$.

For every $C\in \circuits_W$, the elementary vectors with support $C$ form a one-dimensional subspace of $W$. We pick a representative $g^{C,W}\in \EE(W)$ from this subspace.
If $W$ is not a rational subspace, we select $g^{C,W}$ arbitrarily. For rational subspaces, we select $g^{C,W}$ as an integer vector with the largest common divisor of the coordinates being 1;
this choice is unique up to multiplication by $-1$.  When clear from the context, we omit the index $W$ and simply write $g^C$.
We now define the fractional circuit imbalance measure and two variants of integer circuit imbalance measure. 

\begin{Def}[Circuit imbalances]\label{def:kappa} For a non-trivial linear subspace $W\subseteq \R^n$, let us define the following notions:
\begin{itemize}
    \item
The \emph{fractional circuit imbalance measure} of $W$ is
\[
\kappa_W:=\max\left\{\left|\frac{g^C_j}{g^C_i}\right|:\,C\in\circuits_W, i,j\in C\right\}\, .
\]
\item If $W$ is a rational linear space,  the \emph{lcm-circuit imbalance measure}  is
\[
\kd_W:=\lcm\left\{\lcm(g^C):\,C\in\circuits_W\right\}\, .
\]
\item If $W$ is a rational linear space, the 
 \emph{max-circuit imbalance measure}  is
\[
\bar\kappa_W:=\max\left\{\|g^C\|_\infty:\,C\in\circuits_W\right\}\, .
\]
\end{itemize}
For trivial subspaces $W$, we define $\Le_W=\kd_W=\bar\kappa_W=1$.
Further, we say that the rational subspace $W$ is \emph{anchored}, if every vector $g^C$, $C\in\circuits_W$ has a $\pm 1$ entry.
\end{Def} 
Equivalently, in an anchored subspace every elementary vector $g\in\EE(W)$ has a nonzero entry such that all other entries are integer multiples of this entry.

 The term \emph{`circuit imbalance measure'} will refer to the fractional measure $\kappa_W$.
Note that $1\le \kappa_W\le \bar\kappa_W\le\dot\kappa_W$ and $\kappa_W=1$ implies $\bar\kappa_W=\dot\kappa_W=1$. This case plays a distinguished role and turns out to be equivalent to $W$ being a regular linear space (see Theorem~\ref{thm:tu_iff_kappa_1}).

 Another important case is when $\dot\kappa_W=p^\alpha$ is a prime power. In this case, $W$ is anchored, and  $\kappa_W=\bar\kappa_W=\dot\kappa_W$.
The linear space will be often represented as $W=\ker(A)$ for a 
a matrix $A\in \R^{m\times n}$.
We will use  $\EE(A)$, $\circuits_A$, $\Le_A$, $\kd_A$, $\bar\Le_A$ to refer to the corresponding quantities in $\ker(A)$.

\medskip

An earlier systematic study of elementary vectors was done in
Lee's work ~\cite{Lee89}. He mainly focused on the max-circuit imbalance measure; we give a quick comparison to the results in Section~\ref{sec:basic}. The fractional circuit imbalance measure played a key role in the paper \cite{DHNV20} on layered-least-squares interior point methods; it turns out to be  a close proxy to the well-studied condition number $\bar\chi_W$. As far as the authors are aware, the lcm-circuit imbalance measure has not been explicitly studied previously.

\paragraph{Overview and contributions}
Section~\ref{sec:prelim} introduces some background and notation.
Section~\ref{sec:basic} gives an overview of fundamental properties of $\Le_W$ and $\kd_W$. In particular, Section~\ref{sec:subdet} relates circuit imbalances to subdeterminant bounds. 
We note that many extensions of totally unimodular matrices focus on matrices with bounded subdeterminants. Working with circuit imbalances directly can often lead to stronger and conceptually cleaner results.
Section~\ref{sec:HK} presents an extension of the Hoffman-Kruskal characterization of TU matrices. Section~\ref{sec:self-dual} shows an important self-duality property of $\Le_W$ and $\kd_W$. Section~\ref{sec:matrix} studies `nice' matrix representations of subspaces with given lcm-circuit imbalances.
Section~\ref{sec:triangle} proves a multiplicative triangle-inequality for $\Le_W$. Many of these results were previously shown by Lee~\cite{Lee89}, Appa and Kotnyek~\cite{Appa04}, and by Dadush et al.~\cite{DHNV20}. We present them in a unified framework, extend some of the results, and provide new proofs.

Section~\ref{sec:connections} reveals connections between $\Le_W$ and the well-studied condition numbers $\bar\chi$ studied in the context of interior point methods, and $\delta$ studied---among other topics---in the analysis of the shadow simplex method. In particular, we show that previous diameter bounds for polyhedra can be translated to strong diameter bounds in terms of the condition number $\Le_W$ (Theorem~\ref{thm:kappa-diameter}).

Section~\ref{subsec:def_chi_bar_star} studies the best possible values of $\Le_{W}$ that can be achieved by rescaling the variables. We present the algorithm and min-max characterization from \cite{DHNV20}. Further, we characterize when a subspace can be rescaled to a regular one; we also give a new proof of a theorem from~\cite{Lee89}.

Section~\ref{sec:hoffman} shows variants of Hoffman-proximity bounds in terms of $\Le_W$ that will be used in  subsequent algorithms. In Section~\ref{sec:only-A}, we study algorithms for linear programming whose running time only depends on the constraint matrix $A$, and reveal the key role of $\Le_A$ in this context. Section~\ref{sec:black-box} shows how the Hoffman-proximity bounds can be used to obtain a black-box algorithm with $\Le_A$-dependence as in \cite{DadushNV20}, and Section~\ref{sec:LLS} discusses layered least squares interior point methods \cite{DHNV20,Vavasis1996}.

Section~\ref{sec:augment} gives an overview of circuit diameter bounds and circuit augmentation algorithms, a natural class of LP algorithms that work directly with elementary vectors. As a new result, we present an improved iteration bound on the steepest-descent circuit augmentation algorithm, by extending the analysis of the minimum mean-cycle cancelling algorithm of Goldberg and Tarjan (Theorem~\ref{thm:steepest-main}).

Section~\ref{sec:graver} gives an outlook to integer programming, showing the relationship between the max-circuit imbalance and Graver bases. Finally, Section~\ref{sec:conjecture} formulates a conjecture on circuit decompositions with bounded fractionality.

\section{Preliminaries}\label{sec:prelim}

We let $[n] := \{1, \ldots, n\}$. For $k\in\N$, a number $q\in\Q$ is \emph{$1/k$-integral} if it is an integer multiple of $1/k$. Let $\primes\subseteq\N$ denote the set of primes.
Let $\R_{++}$ denote the set of positive reals, and $\R_+$ the set of nonnegative reals.

For a prime number $p\in \primes$, the $p$-\emph{adic valuation} for $\mathbb{Z}$ is the function $\nu_p\colon \Z \to \N$ defined by
\begin{equation}\label{eq:p-adic}
  \nu_p(n) =
  \begin{cases}
      \max \{v \in \N : p^v \divides n\} & \text{if } n \neq 0\\
      \infty & \text{if } n=0.
  \end{cases}
\end{equation}

We denote the support of a vector $x \in \R^n$ by $\supp(x) = \{i\in [n]: x_i \neq 0\}$.  We let ${\1}_n$ denote the $n$-dimensional all-ones vector, or simply $\1$, whenever the dimension is clear from the context. Let $e_i$ denote the $i$-th unit vector.

For vectors $v, w \in \R^n$ we denote by $\min\{v,w\}$ the vector $z \in \R^n$ with $z_i = \min\{v_i, w_i\}, i \in [n]$; analogously for $\max\{v,w\}$. Further, we use the notation $v^+ = \max\{v, 0_n\}$ and $v^- = \max\{-v, 0_n\}$ 
; note that both $v^+$ and $v^-$ are nonnegative vectors. 
For two vectors $x, y \in \R^n$, we let $\pr{x}{y}=x^\top y$ denote their scalar product. For sets $S, T \subseteq \R$ we let $S \cdot T = \{st | s \in S, t \in T\}$.

We let $I_n\in \R^{n\times n}$ denote the $n$-dimensional identity matrix.  We let $\bD_n$ denote the set of all positive definite $n\times n$
diagonal matrices. 
For a vector $v \in \R^n$, we denote by $\Diag(v)$ the diagonal matrix whose $i$-th diagonal entry is $v_i$ and for a matrix $A\in\R^{m\times n}$, let $A_1,A_2,\ldots,A_n\in \R^n$ denote the column vectors, and $A^1,A^2,\ldots,A^m\in\R^m$ denote the row vectors, transposed. For $S\subseteq [n]$, let $A_S$ denote the submatrix formed by the columns of $A$ and for $B\subseteq [n]$, $|B|=m$, we say that $A$ is in \emph{basis form} for $B$ if $A_B=I_m$.

We will use $\ell_1,\ell_2$ and $\ell_\infty$ vector norms, denoted as $\|.\|_1,\|.\|_2$, and $\|.\|_\infty$, respectively. By $\|v\|$, we always mean the 2-norm $\|v\|_2$. Further, for a matrix $A\in\R^{m\times n}$, $\|A\|$ will refer to the $\ell_2\to\ell_2$ operator norm, 
and $\|A\|_{\max}=\max_{i,j}|A_{ij}|$ to the max-norm.

For an index subset $I\subseteq [n]$, we use $\pi_I: \R^n \rightarrow \R^I$ for the coordinate
projection. That is, $\pi_I(x)=x_I$, and for a subset $S\subseteq
\R^n$, $\pi_I(S)=\{x_I:\, x\in S\}$.
We let $\R^n_I = \{x \in \R^n : x_{[n]\setminus I} = 0\}$.

For a subspace $W\subseteq \R^n$, we let $W_I=\pi_I(W\cap \R^n_I)$. It is easy to see that $\pi_I(W)^\perp = (W^\perp)_I$. 
Assume we are given a matrix $A\in \R^{m\times n}$ such that $W=\ker(A)$. Then, $W_I=\ker(A_I)$, and we can obtain a matrix $A'$ from $A$ such that $\pi_I(W)=\ker(A')$ by performing a Gaussian elimination of the variables in $[n]\setminus I$.

For a subspace $W \subseteq \R^n$, we define by $\Pi_W\colon \R^n \to \R^n$ the orthogonal projection onto $W$.

For a set of vectors $V=\{v_i:\,  i\in I\}$ we let $\spn(V)$ denote the linear space spanned by the vectors in $V$. For a matrix $A\in\R^{m\times n}$, $\spn(A)\subseteq\R^{m}$ is the subspace spanned by the columns of $A$.
A \emph{circuit basis} of a subspace $W\subseteq \R^n$ is a set ${\cal F}\subseteq {\EE}(W)$ of $\rk(W)$ linearly independent elementary vectors, i.e., $\spn({\cal F})=W$.

\paragraph{Linear Programming (LP) in matrix formulation} We will use LPs in the following standard primal and dual form for $A\in\R^{m\times n}$, $b\in \R^m$, $c\in \R^n$. 

\begin{equation}
  \label{LP_primal_dual} \tag{LP$(A,b,c)$}
  \begin{aligned}
  \min \; &\pr{c}{x} \quad \\
  Ax& =b \\
  x &\geq 0. \\
  \end{aligned}
  \quad\quad\quad
  \begin{aligned}
  \max \; & \pr{y}{b} \\
  A^\top y + s &= c \\
  s & \geq 0. \\
  \end{aligned}
\end{equation}

\paragraph{Linear Programming  in subspace formulation}
Since our main focus is on properties of subspaces, it will be more natural to think about linear programming in the following \emph{subspace formulation}.
For $A$, $b$ and $c$ as above,  let $W=\ker(A)\subseteq \R^n$ and $d\in \R^n$ such that $Ad=b$. We assume the existence of such a vector $d$ as otherwise the primal program is trivially infeasible. We can write \ref{LP_primal_dual}  in the following equivalent form:
\begin{equation}
\tag{LP$(W,d,c)$}
\label{LP-subspace-f}
\begin{aligned}
\min \; & \pr{c}{x} \\ 
x &\in W + d \\
x &\geq 0.\,
\end{aligned}
\quad\quad\quad
\begin{aligned}
\max \; & \pr{c - s}{d} \\
s &\in W^\perp + c \\
s & \geq 0.
\end{aligned}
\end{equation}

\paragraph{Conformal circuit decompositions}
\label{par:sign_consistent}
We say that the vector $y \in \R^n$ \emph{conforms to}
$x\in\R^n$ if $x_iy_i > 0$ whenever $y_i\neq 0$. 
Given a subspace $W\subseteq \R^n$, a \emph{conformal circuit decomposition} of a vector $z\in W$ is a decomposition
\[
z=\sum_{k=1}^h  g^k,
\]
where
$h\le n$ and  $g^1,g^2,\ldots,g^h\in \EE(W)$ are elementary vectors that are conformal with
$z$. A fundamental result on elementary vectors asserts the existence of a conformal circuit decomposition, see e.g. \cite{Fulkerson1968,RockafellarTheEV}.

\begin{lemma} \label{lem:sign-cons}
For every subspace $W\subseteq \R^n$,  every $z\in W$ admits a conformal circuit decomposition. 
\end{lemma}
\begin{proof}
Let $F \subseteq W$ be the set of vectors conformal with $z$. $F$ is a polyhedral cone; its
 faces  correspond to inequalities of the form $y_k \geq 0$, $y_k\leq 0$,
or $y_k = 0$.
The rays (edges) of $F$ are of the form $\{\alpha g:\,  \alpha\ge 0\}$  for $g\in\EE(W)$. Clearly, $z\in F$, and thus, $z$ can be written as a conic combination of at most $n$ rays by the Minkowski--Weyl theorem. Such a decomposition yields a conformal circuit decomposition.
\end{proof}

\paragraph{Linear matroids}
For a linear subspace $W\subseteq\R^n$, let ${\cal M}(W)=([n],{\cal I})$ denote the associated linear matroid, i.e. the matroid defined by the set of circuits $\circuits_W$. Here, ${\cal I}$ denotes the set of independent sets; $S\in {\cal I}$ if and only if there exists no $z\in W\setminus \{0\}$ with $\supp(z)\subseteq S$; the maximal independent sets are the \emph{bases}.
We refer
the reader to \cite[Chapter 39]{schrijver} or \cite[Chapter
5]{frankbook} for relevant definitions and background on matroid theory.

 Assume $\rk(W)=m$ and $W=\ker(A)$ for $A\in\R^{m\times n}$. Then $B\subseteq [n]$, $|B|=m$ is a basis in ${\cal M}(A):={\cal M}(W)$ if and only if $A_B$ is nonsingular; then, $A'=A_B^{-1}A$ is in basis form for $B$ 
such that $\ker(A')=W$. 

The matroid $\cal M$ is \emph{separable}, if the ground set $[n]$ can be
partitioned into two nonempty subsets $[n]=S\cup T$ such that $I\in
{\cal I}$ if and only if $I\cap S,I\cap T\in {\cal I}$. In this case, the matroid is the direct sum of its
restrictions to $S$ and $T$. In particular, every circuit is fully
contained in $S$ or in $T$. For the linear matroid ${\cal M}(A)$, separability means that $\ker(A)=\ker(A_S)\oplus \ker(A_T)$. 
 In this case, we have $\kappa_A=\max\{\kappa_{A_S},\kappa_{A_T}\}$ and  $\kd_A=\lcm\{\kd_{A_S},\kd_{A_T}\}$; 
solving \ref{LP_primal_dual} can be decomposed into two
subproblems, restricted to the columns in $A_S$ and in $A_T$.

Thus, for most concepts and problems considered in this paper, we can focus on the \emph{non-separable} components of ${\cal M}(W)$.
The following characterization will turn out to be very useful, see e.g.
\cite[Theorem 5.2.5]{frankbook}.

\begin{prop}\label{prop:conn}
A matroid ${\cal M}=([n],{\cal I})$ is non-separable if and only if for 
any $i,j\in [n]$, there exists a circuit containing $i$ and $j$.
\end{prop} 
\section{Properties of the imbalance measures}\label{sec:basic}

\paragraph{Comparison to well-scaled frames}
Lee's work ~\cite{Lee89} on \emph{`well-scaled frames'},
investigated the following closely related concepts. For a set $S\subseteq \Q$ the rational linear space $W$ is \emph{$S$-regular} if for every elementary vector $g\in \EE(F)$, there exists a $\lambda\neq 0$ such that all nonzero entries of $\lambda g$ are in $S$. For $S=\{-k,\ldots,k\}$, the subspace is called \emph{$k$-regular}. 
For $k,\Omega\in \N$, a subspace is \emph{$k$-adic of order $\Omega$} if it is $S$-regular for $S=\{\pm 1,\pm k,\ldots,\pm k^\Omega\}$. The \emph{frame} of the subspace $W$ refers to the set of elementary vectors $\EE(W)$.

Using our terminology, 
a subspace is $k$-regular if and only if $\bar\kappa_W=k$, and every $k$-adic subspace is anchored. Many of the properties in this section were explicitly or implicitly shown in Lee~\cite{Lee89}. However, it turns out that many  properties are simpler and more natural to state in terms of either $\kappa_W$ and $\dot\kappa_W$. Roughly speaking, the fractional circuit imbalance $\kappa_W$ is the key quantity of interest for continuous properties, particularly relevant for proximity results in linear programming. On the other hand, the lcm-circuit imbalance $\dot\kappa_W$ captures most clearly the integrality properties. The max-circuit imbalance $\bar\kappa_W$ interpolates between these two, although, as already noted by Lee, it is the right quantity for proximity results in integer programming (see Section~\ref{sec:graver}).

Appa and Kotnyek \cite{Appa04} also use the term $k$-regularity in a different sense, as a natural extension of unimodularity. This turns out to be strongly related to $\kd_W$; see Lemma~\ref{lem:sub-inverse} and Corollary~\ref{cor:HK-kappa}.

\paragraph{The key lemma on basis forms}
The following simple proposition turns out to be extremely useful in deriving properties of $\Le_W$ and $\kd_W$. The first statement is from \cite{DadushNV20}.
\begin{prop}\label{prop:kappa-max}
For every matrix $A\in \R^{m\times n}$ with $\rk(A) = m$,
\[
    \Le_A = \max\left\{ \|A_B^{-1} A\|_{\max} : A_B \text{ non-singular $m \times m$-submatrix of } A\right\}\, .
\]
Moreover, for each nonsingular $A_B$, all nonzero entries of $A_B^{-1} A$ have absolute values between $1/\Le_A$ and $\Le_A$ and are $1/\kd_A$-integral.
\end{prop} 
\begin{proof}
Consider the matrix $A'=A_B^{-1}A$ for any non-singular $m\times m$ submatrix $A_B$. Let us renumber the columns such that $B$ corresponds to the first $m$ columns. 
Then, for every $m+1\le j\le n$, the $j$th column of $A'$ corresponds to an elementary vector $g$ where $g_j=1$, and $g_i=-A'_{ij}$ for $i\in [m]$. Hence, $\|A'\|_{\max}$ gives a lower bound on $\Le_A$. This also implies that all nonzero entries are between $1/\Le_A$ and $\Le_A$. To see that all entries of $A'$ are $1/\kd_A$-integral, note that $g=g'/\alpha$ for a vector $g'$ where all entries are integer divisors of $\dot\kappa_A$. Since $g_j=1$, it follows that $\alpha$ itself is an integer divisor of $\kd_A$.

To see that the maximum in the first statement is achieved, take the elementary vector $g^C$ that attains the maximum in the definition of $\Le_A$; let $g_j^C$ be the minimum absolute value element.
Let us select a basis $B$ such that $C\setminus \{j\}\subseteq B$. Then, the largest absolute value in the $j$-th column of $A_B^{-1}A$ will be $\Le_A$.
\end{proof}

\subsection{Bounds on subdeterminants}\label{sec:subdet}
For an integer matrix $A\in \Z^{m\times n}$, we define 
\begin{equation}
\begin{aligned}
  \Delta_A &:= \max \{|\det(B)|: B \text{ is a  nonsingular submatrix of } A\},\mbox{ and}\\
  \dot\Delta_A &:=\lcm \{|\det(B)|: B \text{ is a nonsingular submatrix of } A\}.\\
  \end{aligned}
\end{equation} 
The matrix is \emph{totally unimodular (TU)}, if $\Delta_A=1$: thus, all subdeterminants are $0$ or $\pm 1$. This class of matrices plays a foundational role in combinatorial optimization,  see e.g., \cite[Chapters 19-20]{SchrijverLPIP}. A significant example is the node-arc incidence matrix of a directed graph. A key property is that they define integer polyhedra, see Theorem~\ref{thm:HK} below.
A polynomial-time algorithm is known to decide whether a matrix is TU, 
based on the deep decomposition theorem by Seymour from 1980 \cite{Seymour1980}.

The next statement is implicit in \cite[Proposition 5.3]{Lee89}. 
\begin{prop}
  \label{prop:kappa-delta}
  For every integer matrix $A\in \Z^{m\times n}$, $\bar\Le_A\le \Delta_A$ and $\kd_A\le\dot\Delta_A$.
\end{prop}
\begin{proof}
Let  $C\in\circuits_A$ be a circuit, and select a submatrix $\hat A\in \Z^{(|C|-1)\times |C|}$ of $A$ where the columns are indexed by $C$, and the rows are linearly independent. 
Let $\hat A_{-i}$ be the square submatrix resulting from deleting the column corresponding to $i$ from $\hat A$.
From Cramer's rule, we see that $|g^C_i|=|\det(\hat A_{-i})|/\alpha$ for some $\alpha\in\Q$, $\alpha \ge 1$. This implies both claims $\bar\Le_A\le \Delta_A$ and $\kd_A\le\dot\Delta_A$.
\end{proof}

In Propositions~\ref{prop:B_inverse} and \ref{prop:B_inverse_General}, 
we show that for any matrix $A \in \Q^{m \times n}$ there exists a matrix $\tilde A \in \Z^{m \times n}$ such that $\ker(A) = \ker(\tilde A)$ and $\dot\Delta_{\tilde A} \le (\kd_A)^{m}$.

To see an example where $\Delta_A$ can be much larger than $\kappa_A$, let $A\in \Z^{n\times \binom{n}{2}}$ be the node-edge incidence matrix of a complete undirected graph on $n$ nodes; assume $n$ is divisible by $3$. The determinant corresponding to any submatrix corresponding to an odd cycle is $\pm 2$. Let $H$ be an edge set of $\frac{n}3$ node-disjoint triangles. Then $A_H$ is a square submatrix with determinant $\pm 2^{n/3}$. In fact, $\Delta_A=2^{n/3}$ in this case, since $\Delta_A$ for a node-edge incidence matrix equals the maximum number of node disjoint odd cycles, see  \cite{grossmanMinorsIncidenceMatrix1995}.
On the other hand, $\kappa_A=\bar\kappa_A=\kd_A\in \{1,2\}$ for the incidence matrix $A$ of any undirected graph; see Section~\ref{sec:HK}.

\medskip

For TU-matrices, the converse of Proposition~\ref{prop:kappa-delta}  is also true. In 1956, Heller and Tompkins \cite{heller1957,heller1956} introduced the \emph{Dantzig property}. A matrix $A\in\R^{m\times n}$ has the Dantzig property if $A^{-1}_B A$ is a $0,\pm 1$-matrix for every nonsingular $m\times m$ submatrix $A_B$. According to Proposition~\ref{prop:kappa-max}, this is equivalent to $\kappa_A=1$. Theorem~\ref{thm:tu_iff_kappa_1} below can be attributed to Cederbaum \cite[Proposition (v)]{Cederbaum57}; see also Camion's PhD thesis \cite[Theorem 2.4.5(f)]{camion1964}.
The key is the following lemma that we formulate for general $1/\kd_A$ for later use.

\begin{lemma}\label{lem:sub-inverse}
Let $A=(I_m|A')\in\R^{m\times n}$. Then, for any nonsingular square submatrix $M$ of $A$, the inverse $M^{-1}$ is $1/\kd_A$-integral, with non-zero entries between $1/\Le_A$ and $\Le_A$ in absolute value.
\end{lemma}
\begin{proof}
Let $M$ be any $k \times k$ nonsingular submatrix of $A$; w.l.o.g., let us assume that it uses the first $k$ rows of $A$. Let $B$ be the set of columns of $M$, along with the $m-k$ additional columns $i\in [k+1,m]$, i.e., the last $m-k$ unit vectors from $I_m$. Thus,
$A_B\in \R^{m\times m}$ is also nonsingular. After permuting the columns, this can be written in the form
\[
A_B=\left(\begin{array}{c|c}
M & 0  \\
\hline
L& I_{m-k} 
\end{array}\right)
\]
for some $L\in \Z^{(m-k)\times k}$.
We now use Proposition~\ref{prop:kappa-max}  for $
\tilde A = A_B^{-1} A$. 
Note that the first $m$ columns of $
\tilde A$ correspond to $A_B^{-1}$. Moreover, we see that
\[
A_B^{-1}=\left(\begin{array}{c|c}
 M^{-1} & 0 \\
\hline
-LM^{-1} & I_{m-k} 
\end{array}\right)
\]
Thus, $M^{-1}$ is $1/\dot\kappa_A$-integral, with non-zero entries between $1/\Le_A$ and $\Le_A$ completing the proof.
\end{proof}
Appa and Kotnyek define $k$-regular matrices as follows: a rational matrix $A'\in\R^{m\times n}$ is $k$-regular if and only if the inverse of all nonsingular submatrices is $1/k$-integral. From the above statement, it follows that $A'$ is $k$-regular in this sense for $k=\kappa_{(I_m|A')}$. See also
Corollary~\ref{cor:HK-kappa}.

\begin{theorem}[Cederbaum, 1957]
  \label{thm:tu_iff_kappa_1}
  Let $W \subset \R^n$ be a linear subspace. Then, the following are equivalent.
  \begin{enumerate}[(i)]
  \item\label{Ced:i} $\kappa_W=\bar\kappa_W=\kd_W = 1$.
  \item\label{Ced:ii} There exists a TU matrix $A$, such that $W = \ker(A)$.
  \item\label{Ced:iii} For any matrix $A$ in basis form such that $W=\ker(A)$, $A$ is a TU-matrix.
  \end{enumerate} 
\end{theorem}
\begin{proof}
\ref{Ced:iii} $\Rightarrow$ \ref{Ced:ii} is straightforward, and \ref{Ced:ii} $\Rightarrow$ \ref{Ced:i} follows by  Proposition~\ref{prop:kappa-delta}. It remains to show \ref{Ced:i} $\Rightarrow$ \ref{Ced:iii}. Let $\rk(W)=n-m$, and consider any $A\in \R^{m\times n}$ in basis form such that $W=\ker(A)$. For simplicity of notation, assume the basis is formed by the first $m$ columns, that is, 
 $A=(I_{m} | A')$ for some $A' \in \R^{m \times (n-m)}$.

Proposition~\ref{prop:kappa-max} implies that all entries of $A$ are $0$ and $\pm 1$. Consider any nonsingular square submatrix $M$ of $A$.
By Lemma~\ref{lem:sub-inverse}, $M^{-1}$ is also a $0$, $\pm 1$ matrix.
 Consequently, both $\det(M)$ and $\det(M^{-1})$ are nonzero integers, which implies that $|\det(M)|=1$, as required.
\end{proof}

\subsection{Fractional integrality characterization}\label{sec:HK}

Hoffman and Kruskal \cite{Hoffman56} gave the following characterization of TU matrices. A polyhedron $P\subseteq \R^n$ is \emph{integral}, if all vertices (=basic feasible solutions) are integer.
\begin{theorem}[Hoffman and Kruskal, 1956]\label{thm:HK}
An integer matrix $A\in\Z^{m\times n}$ is totally unimodular if and only if for every $b\in\Z^m$
, the polyhedron $\{x\in \R^n:\, Ax\le b,x\ge 0\}$
 is integral.
\end{theorem}

Since $\kd$ is a property of the subspace, it will be more convenient to work with the standard equality form of an LP. Here as well as in Section~\ref{sec:delta}, we use the following straightforward correspondence between the two forms. Recall that an \emph{edge} of a polyhedron is a bounded one dimensional face; every edge is incident to exactly two vertices. The following statement is standard and easy to verify.

\begin{lemma}\label{lem:P-project}
Let $ A\in \R^{m\times n}$ be of the form $A=(A'|I_m)$ for $A'=\R^{m\times (n-m)}$. For a vector $b\in\R^m$, let 
\[P_b=\{x\in\R^n:\, Ax=b, x\ge 0\}\quad\mbox{and}\quad P'_b=\{x'\in\R^{n-m}:\, A'x'\le b, x'\ge 0\}\, .\]
 Let $I=[n-m]$ denote the index set of $A'$. Then,  $P'_b=\pi_I(P_b)$, i.e., $P'_b$ is the projection of $P_b$ to the coordinates in $I$.
For every vertex $x$ of $P_b$,  $x'=x_I$ is a vertex of $P'_b$, and conversely, for every vertex $x'$ of $P'_b$, there exists a unique vertex $x$ of $P$ such that $x_I=x'$. There is a one-to-one correspondence between the edges of $P_b$ and $P'_b$.
Further, if $b\in\Z^m$, then $P_b$ is $1/k$-integral if and only if $P'_b$ is $1/k$-integral.
\end{lemma}

Using Theorem~\ref{thm:tu_iff_kappa_1} and Lemma~\ref{lem:P-project}, we can formulate Theorem~\ref{thm:HK} in subspace language.
\begin{cor}
Let $W\subseteq\R^n$ be a linear space. Then, $\kappa_W=1$ if and only if for every $d\in \Z^n$, the polyhedron $\{x\in\R^n:\, x\in W+d, x\ge 0\}$ is integral.
\end{cor}
\begin{proof}
Let $n'=n-m=\dim(W)$.  W.l.o.g., assume the last  $m$ variables form a basis, and 
let us represent $W$ in a basis form as $W=\ker(A)$ for $A=(A'|I_m)$, where $A'\in\R^{m\times n'}$. It follows by Theorem~\ref{thm:tu_iff_kappa_1} that $\kappa_W=1$ if and only if $A$ is TU, which is further equivalent to $A'$ being TU.

Further, note that the system $\{x\in\R^n:\, x\in W+d, x\ge 0\}$  coincides with  $P_b=\{x\in\R^n:\, (A'|I_m)x=b, x\ge 0\}$, where $b=Ad$.

Note that  $b=Ad$ is integer whenever $d\in\Z^m$. Moreover, we can obtain every integer vector in $b\in\Z^m$ this way, since $A$ contains an identity matrix.
According to Lemma~\ref{lem:P-project}, $P_b$ is integral if and only if $P'_b=\{x\in \R^{n-m}:\, A'x'\le b, x'\ge 0\}$ is integral. The claim follows by Theorem~\ref{thm:HK}.
\end{proof}

We provide the following natural generalization. Related statements, although in substantially more complicated forms, were given in \cite[{Proposition 6.1 and 6.2}]{Lee89}.

\begin{theorem}\label{thm:kappa-HK}
Let $W\subseteq\R^n$ be a linear space. Then, $\dot\kappa_W$ is the smallest integer $k\in\Z$ such that for every $d\in \Z^n$,
the polyhedron $\{x\in\R^n:\, x\in W+d, x\ge 0\}$ is $1/k$-integral.
\end{theorem}
\begin{proof}  Let $\dim(W)=n-m$, and let us represent $W=\ker(A)$ for $A\in\R^{m\times n}$. Then, $x\in W+d$, $x\ge 0$ can be written as $Ax=Ad$, $x\ge 0$. 
Let $x$ be a basic feasible solution (i.e. vertex) of this system. Then, $x=A_B^{-1}Ad$. By Proposition~\ref{prop:kappa-max}, $A_B^{-1}A$ is $1/\kd_W$-integral. Thus, if $d\in \Z^n$ then $x$ must be also $1/\kd_W$-integral.

Let us now show the converse direction. Assume $\{x\in\R^n:\, x\in W+d, x\ge 0\}$ is $1/k$-integral for every $d\in\Z^n$. For a contradiction, assume there exists a circuit $C\in\circuits_W$ such that the entries of the elementary vector  are not all divisors of $k$ (or that $g^C$ is not even a rational vector if $W$ is not a rational space). In particular, select an index $\ell\in C$ such that $g^C_\ell\nmid k$, or such that $(1/g^C_\ell)g^C$is not rational.

 Let us select a basis $B\subseteq [n]$ such that $C\setminus B=\{\ell\}$. For simplicity of notation, let $B=[m]$. We can represent $W=\ker(A)$ in a basis form as $A=(I_m|A')$. Let $g\in\R^n$ be defined by $g_{\ell}=1$, $g_j=-A_{j\ell}$ for $j\in B$ and $g_j=0$ otherwise; thus, $g=(1/g^C_\ell)g^C$. 

Let us pick an integer $t\in\N$, $t\ge \|g\|_\infty$, and define $d\in\Z^n$ by $d_j=t$ for $j\in B$, $d_\ell=-1$, and $d_j=0$ otherwise. Then, the basic solution of $x\in W+d$, $x\ge 0$ corresponding to the basis $B$ is obtained as $x_j=t+g_j$ for $j\in B$ and $x_j=0$ for $j\in [n]\setminus B$. The choice of $t$ guarantees  $x\ge 0$. By the assumption, $x$ is $1/k$-integer, and therefore $g$ is also $1/k$-integer. Recall that $g=(1/g^C_\ell)g^C$, where either $g^C\in\Z^n$ with $\lcm(g^C)=1$ and $g^C_\ell\nmid k$, or $g$ is not rational. Both cases give a contradiction.
\end{proof}

Using again Lemma~\ref{lem:P-project}, we can write this theorem in a form similar to the Hoffman-Kruskal theorem. 

\begin{cor}\label{cor:HK-kappa}
Let $A=(A'|I_m)\in\R^{m\times n}$. Then, $\kd_A$ is the smallest value $k$ such that for 
 every $b\in\Z^m$, the polyhedron $\{x'\in \R^{n-m}:\, A'x'\le b,x'\ge 0\}$
 is $1/k$-integral.
\end{cor} 
Appa and Kotnyek \cite[Theorem 17]{Appa04} show that $k$-regularity of $A'$ (in the sense that the inverse of every square submatrix is $1/k$-integral) is equivalent to the property above. 

\paragraph{Subspaces with $\dot\kappa_A=2$}The case $\kd_W=2$ is a particularly interesting class. As already noted, it includes incidence matrices of undirected graphs, and according to Theorem~\ref{thm:kappa-HK}, it corresponds to half-integer polytopes. This class includes the following matrices, first studied by Edmonds and Johnson \cite{Edmonds1970}; the following result follows e.g. from \cite{Appa04,Gerards1986,Hochbaum1993}.
\begin{theorem}\label{thm:edmonds-johnson}
Let $A\in\Z^{m\times n}$ such that for each column $j\in [n]$, $\sum_{i=1}^m|A_{ij}| \leq 2$. Then $\dot{\kappa}_A\in \{1,2\}$. 
\end{theorem}

Appa and Kotnyek \cite{Appa04} define \emph{binet matrices} as $A'=A_B^{-1} A$ for a matrix $A$ as in Theorem~\ref{thm:edmonds-johnson} for a basis $B$. Clearly, these matrices have $\kd_{A'}\in \{1,2\}$ since they define the same subspace. 

Deciding whether a matrix has $\kd_A=2$ (or more generally, $\kd_A=k$ for a fixed constant $k$) is an interesting open question: is it possible to extend Seymour's decomposition~\cite{Seymour1980} from TU matrices? The matrices in Theorem~\ref{thm:edmonds-johnson} could be a natural building block of such a decomposition.

\subsection{Self-duality}\label{sec:self-dual}
We next show that both $\Le_W$ and $\kd_W$ are self-dual.
These rely on the following duality property of circuits.
We introduce the following more refined quantities that will also come useful later on.

\begin{Def}[Pairwise Circuit Imbalances] \label{def:pairwise circuit_imbalances} For a space $W \subseteq \R^n$ and variables $i,j \in [n]$ we define 
\[
\begin{aligned}\KK_{ij}^W &:= \set{\left|\frac{g_j}{g_i}\right|:\,  \{i,j\} \subseteq \supp(g), g \in \EE(W)}\, ,\quad \Le_{ij}^W := \max \KK_{ij}^W\, ,\\
\dot \KK_{ij}^W &:= \left\{\lcm(p,q) : p,q \in \N, \gcd(p,q) = 1, \frac{p}{q} \in \KK_{ij}^W\right\}\, .
\end{aligned}\]
We call $\Le_{ij}^W $ the \emph{pairwise imbalance between $i$ and $j$}.
\end{Def}
Cleary, $\Le_W=\max_{i,j\in [n]} \Le_{ij}^W$ for a nontrivial linear space $W$.
We use the following simple lemma.
\begin{lemma}\label{lem:basis_circuits}
Consider a matrix $A\in\R^{m\times n}$ in basis form for $B\subseteq[n]$, i.e., $A_B=I_m$. Let $W=\ker(A)$; thus, $W^\perp=\spn(A^\top)$. The following hold.
\begin{enumerate}[(i)]
\item \label{i:basis_circuits-i} The rows of $A$ form a circuit basis of $W^\perp$, denoted as $\EE_B(W^\perp)$.
\item \label{i:basis_circuits-ii} For any two rows $A^i,A^j$, $i,j\in B$, $i\neq j$, and $k\in[n]\setminus B$, the vector $h=A_{jk} A^i-A_{ik} A^j$ fulfills $h \in\EE(W^\perp)$.
\end{enumerate}
\end{lemma}
\begin{proof}
For part \ref{i:basis_circuits-i}, the rows are clearly linearly independent and span $W^\perp$. Therefore,  every $g\in W^\perp$ must have $\supp(g)\cap B\neq\emptyset$, and if $\supp(g)\cap B=\{i\}$ then $g=g_i A^i$. These two facts imply that each $A^i$ is support minimal in $W^\perp$, that is, $A^i\in\EE(W^\perp)$.

For part \ref{i:basis_circuits-ii}, there is nothing to prove if $A_{ik}=0$ or $A_{jk}=0$; for the rest, assume both are nonzero. Assume for a contradiction $h\notin\EE(W^\perp)$; thus, there exists a $g\in W^\perp$, $g\neq 0$ and $\supp(g)\subsetneq \supp(h)$. We have $\supp(h)\cap B=\{i,j\}$. If $\supp(g)\cap B\subsetneq \{i,j\}$, as above we get that $g=g_i A^i$ or $g=g_j A^j$, a contradiction since $h_k=0$ but $A_{ik}, A_{jk}\neq 0$. Hence, $\supp(g)\cap B=\{i,j\}$. By part \ref{i:basis_circuits-i}, we have $g=g_i A^i+g_j A^j$; and since $h_k=0$ it follows that $g_i/g_j=-A_{jk}/A_{ik}$; thus, $g$ is a scalar multiple of $h$, a contradiction.
\end{proof}

\begin{lemma}\label{lem:dual-circuit}
For any $i,j \in [n]$ we have $\KK_{ij}^W = \set{\alpha^{-1} : \alpha \in \KK_{ji}^{W^\perp}}$. Equivalently: for every elementary vector $g \in \EE(W)$ with indices $i,j \in \supp(g)$ there exists an elementary vector $h\in \EE(W^\perp)$  such that $|h_i/h_j| = |g_j/ g_i|$.
\end{lemma}
\begin{proof}
Let $g \in \EE(W)$ such that $i,j \in \supp(g)$. If $\supp(g) = \{i,j\}$ then any $h \in \EE(W^\perp)$ with $i \in \supp(h)$ fulfills $g_ih_i + g_jh_j = \pr{g}{h} = 0$, so $j \in \supp(h)$ and $|h_i/h_j| = |g_j/g_i|$.

Else, there exists $k \in \supp(g) \setminus \{i,j\}$. Let us select a basis $B$  of $\mathcal{M}(W)$ with $\supp(g) \setminus B=\{k\}$. Let $A\in\R^{m\times n}$ be a matrix in basis form for $B$ with $\ker(A)=W$, and let  $h=A_{jk} A^i-A_{ik} A^j$, an elementary vector in $\EE(W^\perp)$
by  Lemma~\ref{lem:basis_circuits}\ref{i:basis_circuits-ii}.

By the construction, $|h_i/h_j|=|A_{jk}/A_{ik}|$. On the other hand,
$\pr{g}{A^i} = 0$ and $\supp(g) \setminus B=\{k\}$ implies $g_i=-g_k A_{ik}$ and similarly $\pr{g}{A^j} = 0$ implies $g_j=-g_k A_{jk}$. The claim follows.
\end{proof}

For $\Le_W$, duality is immediate from the above:
\begin{prop}[\cite{DHNV20}]\label{prop:kappa-dual}
For any linear subspace $W\subseteq \R^n$, we have
$\Le_W=\Le_{W^\perp}$.
\end{prop}

Let us now show duality also for $\kd_W$; this was  shown in \cite[Lemma 2.1]{Lee89} in a slightly different form.
\begin{prop}\label{prop:kappadot-dual}
For any rational linear subspace $W\subseteq \R^n$, we have $\kd_W=\kd_{W^\perp}$.
\end{prop}

\begin{proof}
Recall the $p$-adic valuation $\nu_p(n)$ defined in \eqref{eq:p-adic}.
It suffices to show that $\nu_p(\dot \kappa_W) = \nu_p(\dot \kappa_{W^\perp})$ for any prime $p \in \primes$.
We can reformulate as
\begin{eqnarray*}
\nu_p(\dot \kappa_W) &=& \nu_p(\lcm\left\{\lcm(g^C):\,C\in\circuits_W\right\})\\
&=&\max\left\{\nu_p(\lcm(g^C)):\,C\in\circuits_W\right\}\\
&=& \max \left\{ \nu_p(\alpha) :\,  i,j \in [n], \alpha \in \dot\KK_{ij}^W\right\}. \, 
\end{eqnarray*}
Lemma~\ref{lem:dual-circuit} implies that the last expression is the same for $W$ and $W^\perp$.
\end{proof}

We next show that $\kappa_W$ and $\kd_W$ are monotone under projections and restrictions of the subspace.
\begin{lemma}\label{lem:circuit_minors}
For any linear subspace $W\subseteq \R^n$, $J\subseteq[n]$ and $i,j\in J$, we have 
\[
\KK_{ij}^{\pi_J(W)} \subseteq \KK_{ij}^W\, ,\quad \KK_{ij}^{W_J} \subseteq \KK_{ij}^W\, , \quad \dot\KK_{ij}^{\pi_J(W)} \subseteq \dot\KK_{ij}^W\, ,\quad   \mbox{and}\quad \dot\KK_{ij}^{W_J} \subseteq \dot\KK_{ij}^W\, .
\]
\end{lemma}
\begin{proof}
Let $g \in \EE(W_J)$. Then $(g, 0_{[n]\setminus J} )\in \EE(W)$ and so $\KK_{ij}^{W_J} \subseteq \KK_{ij}^W$. Note that $\pi_J(W) = ((W^\perp)_J)^\perp$ and so by Proposition~\ref{lem:dual-circuit},
\begin{equation}\KK_{ij}^{\pi_J(W)} = \set{\alpha^{-1} : \alpha \in \KK_{ji}^{(W^\perp)_J}} \subseteq \set{\alpha^{-1} : \alpha \in \KK_{ji}^{W^\perp}} = \KK_{ij}^{W}. 
\end{equation} 
The same arguments extend to $\dot\KK_{ij}$.
\end{proof}
\begin{prop}\label{prop:proj-fix}
For any linear subspace $W\subseteq \R^n$ and $J\subseteq[n]$, we have 
\[
\Le_{W_J}\le\Le_W\, ,\quad  \Le_{\pi_J(W)}\le\Le_W\, , \quad \kd_{W_J}\le\kd_W\, ,\quad  \mbox{and}\quad \kd_{\pi_J(W)}\le\kd_W\, .
\]
\end{prop}

\subsection{Matrix representations}\label{sec:matrix}

Proposition~\ref{prop:kappa-max} already tells us that  any rational matrix of the form $A=(I_m|A')$ is $1/\kd_A$-integral, and according to Lemma~\ref{lem:sub-inverse}, the inverse of every non-singular square submatrix of $A$ is also $1/\kd_A$-integral. It is natural to ask whether every linear subspace $W$ can be represented as $W=\ker(A)$ for an 
\emph{integer} matrix $A$ with the same property on the inverse matrices. 

We show that this is true if the dual space is \emph{anchored} but false in general. Recall that this means that every elementary vector $g^C$, $C\in\circuits_{W^\perp}$ has a $\pm 1$ entry. In particular, $\kd_W = p^\alpha$ for some prime number $p \in \primes$ implies that both $W$ and $W^\perp$ are anchored; in this case we also have $\kappa_W=\kd_W$. 

In \cite[Section 7]{Lee89}, it is shown that if $B$ is a basis minimizing $|\det(A_B)|$ for a full rank $A\in \R^{m\times n}$, then every nonzero entry in $A_B^{-1} A$ is at least 1 in absolute value. Moreover, a simple greedy algorithm is proposed (called \textsc{1-OPT}) that finds such a basis within $m$ pivots for $k$-adic spaces. Our next statement can be seen as the variant of this for anchored-spaces, using the lcm-circuit imbalance $\kd_A$.
We note that finding a basis minimizing $|\det(A_B)|$ is computationally hard in general \cite{khachiyan1995}.

\begin{prop} \label{prop:B_inverse}
  Let $W \subseteq \R^n, \dim(W) = n - m$ be a rational subspace such that $W^\perp$ is an  anchored space. Then there exists an integer matrix $A\in \Z^{m\times n}$ such that $\ker(A) = W$, and
  \begin{enumerate}[(i)]
    \item\label{part:norm-bound-p} All entries of $A$ divide $\kd_W$.
    \item\label{part:inverse-p} For all non-singular submatrices $M$ of $A$, $M^{-1}$ is $\frac{1}{\kd_W}$-integral.
    \item\label{part:subdet-p} $\dot\Delta_A$ is an integer divisor of $(\kd_W)^m$.
  \end{enumerate}
\end{prop}
\begin{proof}
  Let $\bar A \in \Q^{m\times n}$ be an arbitrary matrix with $\ker(\bar A) = W$. By performing row operations we can convert $\tilde A$ into $A = (D | A') \in \Z^{m \times n}$ where $D \in \bD_m$ is positive diagonal and $A' \in \Z^{m \times (n-m)}$ (after possibly permuting the columns). If $D=I_m$, then we are already done. Property~\ref{part:norm-bound-p} follows by Proposition~\ref{prop:kappa-max}; property~\ref{part:inverse-p}
 follows by Lemma~\ref{lem:sub-inverse}, and property~\ref{part:subdet-p} holds since $\det(M)\cdot\det(M^{-1})=1$, $\det(M)\in\Z$, and $\det(M^{-1})$ is $\frac{1}{(\kd_W)^m}$-integral.

  If $D$ is not the identity matrix, then we show that $A$ can be brought to the form $(I_m | A'')$ with an integer $A''$ by performing further basis exchanges. 
  Let  us assume that $\gcd(A^i) = 1$ for all rows $A^i$, $i \in [m]$.
By Lemma~\ref{lem:dual-circuit}, $A^i\in \EE(W^\perp)$.
Assume $D_{ii}=A_{ii}> 1$ for some $i\in [m]$.
  As $A^i$ is a circuit and $W^\perp$ is anchored, there exists an index $k \in [n]$ such that $|A_{ik}| = 1$.

Let us perform a basis exchange between columns $i$ and $k$. That is, 
subtract integer multiples of row $i$ from the other rows to turn column $k$ into $e_i$. We then swap columns $i$ and $k$ and obtain the matrix again in the form $(D'|A'')$. 
Notice that the matrix remains integral,  $D'_{ii}=1$, and $D'_{jj}=D_{jj}$ for $j\in [m]$, $j\neq i$. Hence, repeating this procedure at most $n$ times, we can convert the matrix to the integer form $(I_m|A')$, completing the proof.
\end{proof}
Note that the proof gives an algorithm to find such a basis representation using a Gaussian elimination and at most $m$ additional pivot operations.
If $W^\perp$ is not anchored, we show the following weaker statement.
\begin{prop} \label{prop:B_inverse_General}
  Let $W \subseteq \R^n, \dim(W) = n - m$ be a rational subspace. Then there exists an integer matrix $A\in\Z^{m\times n}$ with $\ker(A) = W$ such that
  \begin{enumerate}[(i)]
    \item\label{part:norm-bound} All entries of $A$ divide $\kd_W$;
    \item\label{part:inverse} For all non-singular submatrices $M$ of $A$, $M^{-1}$ is $\frac{1}{(\kd_W)^2}$-integral.
    \item\label{part:subdet} $\dot\Delta_A$ is an integer divisor of $(\kd_W)^{m}$.
  \end{enumerate}
\end{prop}
\begin{proof}
The proof is an easy consequence of Proposition~\ref{prop:kappa-max} and Lemma~\ref{lem:sub-inverse}. Consider any basis form $A=(I_m|A')$ with  $\ker(A)=W$ (after possibly permuting the columns). According to Proposition~\ref{prop:kappa-max},  all entries of $A$ are $1/\dot\kappa_W$ integral.
By Lemma~\ref{lem:dual-circuit}, the rows $A^i\in\EE(W^\perp)$ for $i\in [m]$. We can write $A^i=g^i/d_i$ for some  $g^i\in \EE(W^\perp)\cap \Z^m$ and $d_i\in \Q$ such that $\gcd(g^i)=1$ for each $i\in [m]$. By the definition of $\kd_A$, the entries of each $g^i$ are divisors of $\kd_A$. Since $A_{ii}=1$ it follows that $d_i\in \Z$ and $d_i\divides \kd_A$. Let $D\in\bD_m$ be the diagonal matrix with entries $D_{ii}=d_i$. Then, $\bar A=DA$ is an integer matrix where all entries divide $\kd_A$, proving \ref{part:norm-bound}.
Part~\ref{part:inverse} follows  by Lemma~\ref{lem:sub-inverse} and noting that the subdeterminants get multiplied by a submatrix $D^{-1}$.

For part~\ref{part:subdet}, let us start use a basis $B$ such that $|\det(A_B)|$ is \emph{maximal}; w.l.o.g. assume $B=[m]$. 
Then, in the basis form $(I_m|A')$ for $B$, all subdeterminants are $\le 1$. This holds as for any submatrix $M \in \Q^{k \times k}$ of $A'$ with $\det(M) \neq 0$ we have that augmenting the columns of $M$ by the columns $i \in B$ such that $i$ is not a row of $M$ results in a basis $B_M$ with $|\det(M)| = |\det\big((I_m | A')_{B_M}\big)| \le \det(I_m) = 1$ by assumption on $B$. After multiplying by $D$ as above, $\bar A=DA$, all subdeterminants will be $\le \det(D)\le (\kd_A)^m$.
\end{proof}
Note that parts~\ref{part:norm-bound} and \ref{part:inverse} are true for any choice of the basis form, whereas \ref{part:subdet} requires one to select $A_B$ with maximum determinant. The maximum subdeterminant is NP-hard even to approximate better than $c^m$ for some $c>1$ \cite{DiSumma2014}. However,
it is easy to see that even if we start with an arbitrary basis, then $\dot\Delta_A\divides(\kd_W)^{2m}$, since every subdeterminant of $A_B^{-1} A$ is at most $(\kd_W)^m$ follows by Lemma~\ref{lem:sub-inverse}.

We now give an example to illustrate why Proposition~\ref{prop:B_inverse}\ref{part:inverse-p} cannot hold for arbitrary values of $\kd_W$. The proof is given in the Appendix.
\begin{restatable}{prop}{counterex}\label{prop:counterexample}
Consider the matrix 
$$A = \begin{bmatrix} 1 & 3 & 4 & 3 \\ 0 & 13 & 9 & 10 \end{bmatrix}\, .$$
For this matrix
$\dot{\kappa}_A = 5850 = 2 \times 3^2 \times 5^2 \times 13$ holds, and 
there exists no $\tilde A\in\Z^{2\times 4}$ such that $\ker(\tilde A)=\ker(A)$ and the inverse of every nonsingular $2\times 2$ submatrix of $\tilde A$ is $1/5850$-integral.
\end{restatable}

\subsection{The triangle inequality}\label{sec:triangle}
An interesting additional fact about circuit imbalances is that the logarithm of the weights satisfy the triangle inequality; this was shown in \cite{DHNV20}. Here, we formulate a stronger version and give a simpler proof. Throughout, we assume that ${\cal M}(W)$ is non-separable. Thus, according to Proposition~\ref{prop:conn}, for any $i,j\in [n]$ there is a circuit $C\in \circuits_W$ with $i,j\in C$.
\begin{theorem}\label{thm:imbalance_triangle_inequality}
Let $W\subseteq \R^n$ be a linear space, and assume ${\cal M}(W)$ is non-separable. Then, 
  \begin{enumerate}[(i)]
  \item\label{i:set_triangle_inequality} for any distinct $i,j,k\in [n]$,
  $\KK_{ij}^W \subseteq \KK_{ik}^W \cdot\KK_{kj}^W$; and 
  \item \label{i:triangle} for any distinct $i,j,k\in [n]$,
  $\Le_{ij} \leq \Le_{ik}\cdot \Le_{kj}$.
  \end{enumerate}
  \end{theorem}

The proof relies on the following technical lemma that analyzes the scenario when almost all vectors in $W$ are elementary.
\begin{lemma}
  \label{lem:helper_triangle}
   Let $W \subseteq \R^n$ be a subspace s.t.\ $\mathcal{M}(W)$ is non-separable. 
  \begin{enumerate}[(i)]
  \item \label{it:prop_helper_second} If $\EE(W) = \set{g \in W\setminus\{0\} : \supp(g)\neq [n]}$, then $\KK_{ij}^W \subseteq \KK_{ik}^W\cdot \KK_{kj}^W$. 
    \item \label{it:prop_helper_first} If there exists $g \in \EE(W)$ such that $|\supp(g)| = n - 1$, then 
$$\EE(W) = \set{g \in W\setminus\{0\} : \supp(g)\neq [n]}\, .$$
\end{enumerate}
\end{lemma}
\begin{proof}
  For part \ref{it:prop_helper_second}, 
let $\delta \in \KK_{ij}^W$ and let $g \in \EE(W)$ such that $\{i,j\} \subset \supp(g)$ and $|g_j/g_i| = \delta$. If $k \in \supp(g)$, then $|g_j/g_i| = |g_k/g_i| \cdot |g_j/g_k|$ shows the claim.

Assume $k\notin \supp(g)$, and pick $h \in \EE(W)$ such that $\{i,k\} \subset \supp(h)$ and let $\tilde h = h_jg - g_jh$; such a $h$ exists by Proposition~\ref{prop:conn}. Then $\tilde h_j = 0$ and $\tilde h_k \neq 0$, so $\tilde h \in \EE(W)$ by the assumption. If $\tilde h_i = 0$ then $h_jg_i = g_jh_i$ and so $\{i,j,k\} \subset \supp(h)$ with $h_j/h_i = g_j/g_i$, therefore $h$ certifies the statement as $|h_j/h_i| = |h_k/h_i| \cdot |h_j/h_k|$.
Otherwise, $\tilde h_i \neq 0$ and $h':= \tilde h_ig - g_i\tilde h$ fulfills $h' \in \EE(W)$ as $h'_i = 0$, $\{j, k\} \subset \supp(h')$.
Now, using that $\tilde h_j = 0$ and $g_k = 0$ it is easy to see that
\begin{equation}
  \left|\frac{\tilde h_k}{\tilde h_i}\cdot \frac{h'_j}{h'_k} \right|
  = \left| \frac{\tilde h_k}{\tilde h_i}\cdot \frac{\tilde h_i g_j - g_i \tilde h_j}{\tilde h_i g_k - g_i\tilde h_k} \right|
  =  \left| \frac{\tilde h_k}{\tilde h_i}\cdot \frac{\tilde h_i g_j}{g_i\tilde h_k} \right|
   = \left| \frac{g_j}{g_i} \right| \, .
\end{equation}
We now turn to part \ref{it:prop_helper_first}. Since there exists $g \in \EE(W)$ with $\supp(g) \neq n$, we cannot have $[n] \in \circuits(W)$. 

Let $g \in \EE(W)$ and $i \in [n]$ such that $\supp(g) = [n] \setminus \set{i}$. Consider any $h \in W$, $\supp(h) \neq \supp(g)$ such that $\supp(h) \neq [n]$. If $h \notin \EE(W)$ there exists $\ell \in \EE(W)$ such that $\supp(\ell) \subsetneq \supp(h)$. We must have $i \in \supp(\ell)$, since $\supp(\ell)\setminus \supp(g)\neq \emptyset$. Then $\tilde h: = h_i\ell - \ell_ih$ fulfills $\tilde h\neq 0$, $\tilde h_i = 0$ and $\supp(\tilde h) \subsetneq [n]\setminus \set{i}$, a contradiction to $g \in \EE(W)$.
\end{proof}

\begin{proof}[Proof of Theorem~\ref{thm:imbalance_triangle_inequality}]
  \label{proof:lem:imbalance_triangle_inequality}
Part \ref{i:triangle} immediately follows from part \ref{i:set_triangle_inequality}, when taking
 $C \in \circuits_W$ such that $|g_j^C/g_i^C| = \Le_{ij}$. We now prove part \ref{i:set_triangle_inequality}.
 
 Let $\delta \in \KK_{ij}^W$ and $C\in \circuits_W$ such that $i,j\in C$ and for $g=g^C$, $|g_j/g_i| = \delta$. If $k \in C$ then clearly $|g_j/g_i| = |g_k/g_i| \cdot |g_j/g_k| \in \KK_{ik}^W\cdot \KK_{kj}^W$. Otherwise,
 let us select $C'\in \circuits_W$ such that $i,k\in C'$, and $|C\cup C'|$ is minimal. Let $h=g^{C'}$ and $J=C'\setminus (C\cup\{k\})$.
 
\begin{claim}
Let $G = (C \cup C') \setminus J$. Then for the space
$\hat W := \pi_{G}(W_{C \cup C'})$ we have that $g_G, h_G \in \EE(\hat W)$. 
\end{claim}
\begin{proof}
The statement that $h_G \in \EE(\hat W)$ is clear as $h_{C \cup C'} \in \EE (W_{C \cup C'})$ and the variables we project out $J$ fulfill $J \subseteq \supp(h)$.
For the statement on $g_G$ assume that there exists $\hat g \in \EE(\hat W)$ such that $\supp(\hat g) \subsetneq \supp(g_G)$. Then there exists a lift  $\tilde g \in \EE(W_{C \cup C'})$ of $\hat g$ and some $\ell \in J$ such that $\ell \in \supp(\tilde g)$; note also that $\tilde g_k = g_k = 0$. The vector $\hat h := h_\ell \tilde g - \tilde g_\ell h$ fulfills $\ell \notin \supp(\hat h)$ and $k \in \supp(\hat h)$.

Now pick any circuit $\tilde h \in \EE(\hat W)$ such that $k \in \supp(\tilde h)$ and $\supp(\tilde h) \subseteq \supp(\hat h)$. 
Note that $J \cup \{k\}$ is independent, as $J \cup \{k\} \subseteq C' \setminus \{i\} \subsetneq C'$. Therefore, $\supp(\tilde h) \cap \supp(g) \neq \emptyset$.
Hence, for $T := C \cup \supp(\tilde h)$ we have that $\mathcal M(W_T)$ is non-separable. In particular there exists a circuit $h' \in \EE(W_T)$ such that $i,k \in \supp(h')$. As $T \subseteq (C \cup C')\setminus \{\ell\}$, this is a contradiction to the minimal choice of $C'$.  
\end{proof}
As $\supp(h_D) \cup \supp(g_D) = D$ and $\supp(h_D) \cap \supp(g_D) \neq \emptyset$ we have that $\EE(W')$ is non-separable. Further $|\supp(g_D)| = |D| - 1$, so we can apply Lemma~\ref{lem:helper_triangle} to learn $\delta \in \KK_{ij}^{W'} \subseteq \KK_{ik}^{W'} \cdot \KK_{kj}^{W'}$. We can conclude   $\delta \in \KK_{ik}^W \cdot \KK_{kj}^W$ from Lemma~\ref{lem:circuit_minors}.
\end{proof}

If $\Le_W=1$, then the reverse inclusion 
$ \KK_{ik}^W \cdot \KK_{kj}^W \subseteq \KK_{ij}^W$ trivially holds, since 1 is the only element in these sets. In Proposition~\ref{prop:reverse-kappa-1}, we give a necessary and sufficient condition for $\KK_{ij}^W= \KK_{ik}^W \cdot \KK_{kj}^W$.

One may ask under which circumstances an element $\alpha \in \KK_{ik}^W\cdot \KK_{kj}^W$ is also contained in $\KK_{ij}^W$. We give a partial answer by stating a sufficient condition in a  restrictive setting. For a basis $B$ of ${\cal M}(W)$, recall $\EE_B(W^\perp)$ from Lemma~\ref{lem:basis_circuits}. Then, Lemmas~\ref{lem:basis_circuits} and \ref{lem:dual-circuit} together imply:

\begin{lemma}
Given a basis $B \subseteq [n]$ in $\mathcal{M}(W)$ and $g,h \in \EE_B \subseteq \EE(W^\perp)$ such that $i \in \supp(g) \cap B$, $j \in \supp(h) \cap B$ and $k \in \supp(g) \cap \supp(h)$.
Then $|h_j/h_k|\cdot |g_k/g_i| \in \KK_{ij}^W$.
\end{lemma} 

\section{Connections to other condition numbers}\label{sec:connections}
\subsection{The condition number \texorpdfstring{$\bar\chi$}{chi bar} and the lifting operator}
\label{subsec:lifting_operator}

For a full row rank matrix $A\in \R^{m\times n}$, the
 condition number $\bar\chi_A$ can be defined in the following two equivalent ways:
\begin{equation}
\begin{aligned}
\bar\chi_A&=\sup\left\{\left\|A^\top \left(A D A^\top\right)^{-1}AD\right\|\, : D\in
    \bD_n\right\}\\
&=\sup\left\{\frac{\norm{A^\top y}}{\norm{p}} :
\text{$y$ minimizes $\norm{D^{1/2}(A^\top y - p)}$ for some $0 \neq p \in \R^n$ and $D \in \bD_n$}\right\}.
\end{aligned}
\end{equation}
This condition number was first studied by Dikin \cite{dikin}, Stewart \cite{stewart}, and
Todd \cite{todd-90}. There is an
extensive literature on the properties and applications of $\bar\chi_A$,
as well as its
relations to other condition numbers. In particular, it plays 
a key role in layered-least-squares  interior point methods, see Section~\ref{sec:LLS}.  We refer the reader to
the papers \cite{ho2002,Monteiro2003,Vavasis1996} for further results
and references.

It is important to note that---similarly to $\kappa_A$ and $\dot\kappa_A$---$\bar\chi_A$ only depends on the subspace
$W=\ker(A)$. Hence, we can also write $\bar\chi_W$ for a subspace
$W\subseteq \R^n$, defined to be equal to $\bar\chi_A$ for some matrix
$A\in \R^{k\times n}$ with $W=\ker(A)$. We will use the notations $\bar\chi_A$ and
$\bar\chi_W$ interchangeably. The following characterization reveals the connection between $\kappa_A$ and $\bar\chi_A$.

\begin{prop}[\cite{Todd2001}]\label{prop:chibar}~
For a full row rank matrix $A\in \R^{m\times n}$, 
\[\bar\chi_A = \max\{ \|A_B^{-1} A\| : A_B \text{ is a non-singular $m \times m$-submatrix of } A\} \, .\]
\end{prop}
Together with Proposition~\ref{prop:kappa-max}, this shows that the difference between  $\bar\chi_A$ and $\kappa_A$ is in using $\ell_2$ instead of $\ell_\infty$ norm.
This immediately implies the upper bound and a slightly weaker lower bound in the next theorem.

\begin{theorem}[\cite{DHNV20,DadushNV20}]\label{theorem:kappa_chi_bar_approx}
For a matrix $A \in \R^{m \times n}$ we have $\sqrt{1 + \kappa_A^2} \le \bar\chi_A \le n\kappa_A$.
\end{theorem}

Approximating the condition number $\bar\chi_A$ is known to be hard; by the same token, $\bar\kappa_A$ also cannot be approximated by any polynomial factor. The proof relies on the hardness of approximating the minimum subdeterminant by Khachiyan \cite{khachiyan1995}.
\begin{theorem}[{Tun{\c{c}}el \cite{Tuncel1999}}]\label{thm:chi-inapprox}
  Approximating $\bar\chi_A$ up to a factor of $2^{\operatorname{poly}(n)}$ is NP-hard.
\end{theorem}

\medskip

In connection with $\bar\chi_A$, it is worth mentioning the  {\em lifting map}, a key concept in the algorithms presented in Section~\ref{sec:only-A}. The map 
$L_I^W : \pi_{I}(W) \to W$ lifts back a vector from a coordinate projection of $W$ to a minimum-norm vector in $W$:
\[
L_I^W(p) = \arg\min\left\{\|z\| : z_I = p, z \in W\right\}.
\]
Note that $L_I^W$ is the unique linear map from $\pi_{I}(W)$ to $W$ such that
$L_I^W(p)_I = p$ and $L_I^W(p)$ is orthogonal to $W \cap \R^n_{[n] \setminus
I}$. 
The condition number $\bar\chi_W$ can be equivalently defined as the maximum norm of any lifting map for an index subset.  
\begin{prop}[\cite{DHNV20,OLeary1990,stewart}]\label{prop:subspace-chi}
  For a linear subspace $W  \subseteq \R^n$,
  \[
  \bar{\chi}_W =\max\left\{ \|L_I^W\|\, : {I\subseteq [n]}, I\neq\emptyset\right\}\, .
  \]
\end{prop}

Even though $L_I^W$ is defined with respect to the $\ell_2$-norm, it can also be used to characterize $\kappa_W$.
\begin{prop}[\cite{DadushNV20}]\label{lem:kappa-lift}
For a linear subspace $W  \subseteq \R^n$,
\[
\kappa_W =\max\left\{ \frac{\|L_I^W(p)\|_\infty}{\|p\|_1}\, : {I\subseteq [n]}, I\neq\emptyset, p\in \pi_I(W)\setminus\{0\}\right\}\, .
\]
\end{prop}

\begin{proof}
We first show that for any $I\neq\emptyset$, and $p\in \pi_I(W)\setminus\{0\}$, $\|L_I^W(p)\|_\infty\le \kappa_W \|p\|_1$ holds. 
Let $z=L_I^W(p)$, and take a conformal decomposition $z=\sum_{k=1}^h g^k$ as in Lemma~\ref{lem:sign-cons}. For each $k\in [h]$, let $C_k=\supp(g^k)$. We claim that all these circuits must intersect $I$. Indeed, assume for a contradiction that one of them, say $C_1$ is disjoint from $I$, and let $z'=\sum_{k=2}^h g^k$. Then, $z'\in W$ and $z'_I=z_I=p$. Thus, $z'$ also lifts $p$ to $W$, but $\|z'\|_2<\|z\|_2$, contradicting the definition of $z=L_I^W(p)$ as the minimum-norm lift of $p$.

By the definition of $\kappa_W$, $\|g^k\|_\infty\le \kappa_W \|g^k_I\|_1$ for each $k\in [h]$. The claim follows since $p=z_I=\sum_{k=1}^h g^k_I$, moreover, conformity guarantees that $\|p\|_1=\sum_{k=1}^h \|g^k_I\|_1$. Therefore,
\[
\|z\|_\infty\le \sum_{k=1}^h \|g^k\|_\infty\le \kappa_W \sum_{k=1}^h \|g^k_I\|_1=\kappa_W\|p\|_1\, .\]
We have thus shown that the maximum value in the statement is at most $\kappa_W$. To show that equality holds, let $C\in\circuits_W$ be the circuit and $g^C\in W$ the corresponding elementary vector and $i,j\in C$ such that $\kappa_W=|g^C_j/g^C_i|$.

Let us set $I=([n]\setminus C)\cup \{i\}$, and define $p_k=0$ if $k\in [n]\setminus C$ and $p_i=g^C_i$. Then  $p\in \pi_I(W)$,  and the unique extension to $W$ is $g^C$; thus, $L_I^W(p)=g^C$. We have $\|L_I^W(p)\|_\infty=|g^C_j|$. Noting that $\|p\|_1=|g^C_i|$, it follows that 
$\kappa_W=\|L_I^W(p)\|_\infty/\|p\|_1$.
\end{proof}

\subsection{The condition number $\delta$ and bounds on diameters of polyhedra}\label{sec:delta}

Another related condition number is $\delta$, defined as follows:
\begin{Def}\label{def:delta}
Let $V\subseteq \R^n$ be a set of vectors. Then
$\delta_V$ is the largest value such that for any set of linearly independent vectors $\{v_i:\,  i\in I\}\subseteq V$ and $\lambda\in \R^I$,
\[
\left\|\sum_{i\in I}\lambda_i v_i\right\|\ge \delta_V \max_{i \in I} |\lambda_i| \cdot\|v_i\|\, .
\]
For a matrix $M\in\R^{m\times n}$, we let $\delta_M$ denote the value associated with the rows $M^1,M^2,\ldots,M^m$ of $M$.
\end{Def}
This can be equivalently characterized as follows: for a subset  $\{v_i:\,  i\in I\}\subseteq V$ and $v_j \in V$,  $v_j\notin W=\spn(\{v_i:\,  i\in I\})$,   the sine of the angle between the vector  $v_j$ and the subspace $W$ is at least $\delta_V$ (see e.g. for the equivalence \cite{Dadush-oracle}).

A line of work studied this condition number in the context of the simplex algorithm and diameter bounds. The \emph{diameter} of a polyhedron $P$ is the diameter of the vertex-edge graph associated with $P$; Hirsch's famous conjecture from 1957 asserted that the diameter of a polytope (a bounded polyhedron) in $n$ dimensions with $m$ facets is at most $m-n$. This was disproved by Santos in 2012 \cite{santos2012}, but the \emph{polynomial Hirsch conjecture}, i.e., a poly$(n,m)$ diameter bound remains wide open.

Consider the LP in standard inequality form with $n$ variables and $m$ constraints as 
\begin{equation}\label{LP:ineq}
\max \pr{c}{x}\, \st\,  x\in P\, ,\quad P=\{x\in \R^n:\, Ax\le b\}\, ,
\end{equation}
for $A\in \R^{m\times n}$, $b\in \R^m$.
Using a randomized dual simplex algorithm, Dyer and Frieze \cite{Dyer1994} showed the polynomial Hirsch conjecture for TU matrices. Bonifas et al.~\cite{Bonifas2014} strengthened and extended this to the bounded subdeterminant case, showing a diameter bound of $O(n^4\Delta_A^2 \log(n\Delta_A))$ for integer constraint matrices $A\in\Z^{m\times n}$. Note that this is independent of the number of constraints $m$.

Brunsch and R\"oglin \cite{brunsch2013} analyzed the shadow vertex simplex algorithm in terms of the condition number $\delta_A$, noting that for integer matrices $\delta_A\ge 1/(n\Delta_A^2)$. They gave a diameter bound $O(mn^2/\delta^2_A)$.
 Eisenbrand and Vempala \cite{eisenbrand2017} used a different approach to derive a bound poly$(n,1/\delta_A)$ that is independent of $m$.  Dadush and H\"ahnle \cite{dadush2016shadow} further improved these bounds to $O(n^3\log(n/\delta_A)/\delta_A)$.

 In recent work, Dadush et al.~\cite{Dadush-oracle} considered \eqref{LP:ineq} in the oracle model, where for each point $x\in \R^n$, the oracle returns $x\in P$ or a violated inequality $\pr{a_i}{x}\le b_i$ from the system $Ax\le b$. Their algorithm finds exact primal and dual solutions using $O(n^2\log(n/\delta_M))$ oracle calls, where 
 $M=\begin{pmatrix}\mathbf{0}& 1\\
A & b
\end{pmatrix}$; the running time is independent of the cost function $c$. They also show the following relation between $\kappa$ and $\delta$:

\begin{lemma}[{\cite{Dadush-oracle}}]\label{lem:delta-kappa}
\begin{enumerate}[(i)]
\item \label{i:delta-kappa-part-i} Let $A\in\R^{m\times n}$ be a matrix with full row rank and $m<n$, 
with $\|a_i\|=1$ for all columns $i\in[n]$. Then, $\kappa_A\le{1}/{\delta_{A^\top}}$. 
\item \label{i:delta-kappa-part-ii}Let $A\in\R^{m\times n}$ be in basis form  $A=(I_m|A')$. Then,
${1}/{\delta_{A^\top}}\le m\kappa_A^2$.
\item\label{i:m-kappa} If $B$ is the basis maximizing $|\det(A_B)|$, then  for $\bar A=A_B^{-1}A$, it holds that ${1}/{\delta_{\bar A^\top}}\le m\kappa_A$.
\end{enumerate}
\end{lemma}
\begin{proof}
{\bf Part \ref{i:delta-kappa-part-i}:} Let  $g\in\EE(A)$ be an elementary vector. Select an arbitrary $i\in \supp(g)$, and let $J=\supp(g)\setminus \{i\}$. Then, the columns $\{g_i:\,  i\in J\}$ are linearly independent, and $-g_i a_i=\sum_{j\in J} g_j a_j$. Thus,
\[
|g_i|\cdot\|a_i\|=\left\|\sum_{j\in J} g_j a_j\right\|\ge \delta_{A^\top} \max_{j\in J} |g_j|\cdot\|a_j\|\, ,
\]
and using that all columns have unit norm, we get $|g_j/g_i|\le 1/\delta_{A^\top}$ for all $j\in J$. This shows that $\kappa_{A}\le 1/\delta_{A^\top}$.

\paragraph{Parts \ref{i:delta-kappa-part-ii} and \ref{i:m-kappa}:} Let $A=(I_m|A')$ in basis form, and let $\alpha=\max_{i\in [n]} \|A_i\|$. Let us first show 
\begin{equation}\label{eq:delta-kappa-alpha}
{1}/{\delta_{A^\top}}\le \sqrt{m}\alpha \kappa_A\, .\end{equation}
 Take any set $\{A_i:\, i\in I\}$ of linearly independent columns of $A$, along with coefficients $\lambda\in \R^I$. Without loss of generality, assume $|I|=m$, i.e., $I$ is a basis, by allowing $\lambda_i=0$ for some coefficients. Let $z=\sum_{i\in I}\lambda_i A_i$. Then, $\lambda=A_I^{-1} z$. Lemma~\ref{lem:sub-inverse} implies that every column of $A_I^{-1}$ has 2-norm at most $\sqrt{m}\kappa_A$. Hence, $|\lambda_i|\le \sqrt{m}\kappa_A \|z\|$ holds for all $i\in I$, implying \eqref{eq:delta-kappa-alpha}.

 Then, part \ref{i:delta-kappa-part-ii} follows since $\|A_i\|\le \sqrt{m}\kappa_A$ by Proposition~\ref{prop:kappa-max}. For part \ref{i:m-kappa}, let  $B$ be a basis maximizing $|\det(A_B)|$. Then, $\|A_B^{-1} A\|_\infty\le 1$. Indeed, if there is an entry $|A_{ij}|>1$, then we can obtain a larger determinant by exchanging $i$ for $j$. This implies $\alpha\le \sqrt{m}$.
\end{proof}

Using this correspondence between $\delta$ and $\Le$, we can derive the following bound on the diameter of polyhedra in standard form from \cite{dadush2016shadow}. This verifies the polynomial Hirsch-conjecture whenever $\kappa_A$ is polynomially bounded.

\begin{theorem}\label{thm:kappa-diameter}
Consider a polyhedron in the standard equality form 
\[P=\{x\in\R^n:\, Ax=b, x\ge 0\}\]
for $A\in\R^{m\times n}$ and $b\in \R^m$. Then, the diameter of $P$ is at most  $O((n-m)^{3} m \kappa_{A}$ $\log(\kappa_{A}+n))$.
\end{theorem}
\begin{proof}
Without loss of generality, we can assume that $A$ has full row rank. Changing to a standard basis representation does neither change the geometry (in particular, the diameter) of $P$, nor the value of $\kappa_A$.
Let $B$ be the basis maximizing $\det(A_B)$, and let us replace $A$ by $A_B^{-1} A$; w.l.o.g. assume that $B$ is the set of the last $m$ columns.
 Hence, $A=(A'|I_m)$ for $A'\in \R^{m\times (n-m)}$. 
According to Lemma~\ref{lem:P-project}, $P$ has the same diameter as $P'$ defined as
\[P'=\{x'\in\R^{n-m}:\, A'x'\le b, x'\ge 0\}\, ,\]
in other words, $P'=\{x'\in\R^{n-m}:\, Cx'\le d\}$, where $C={-I_{n-m}\choose A'}$ and $d={0 \choose b}$. There is a one-to-one correspondence between the vertices and edges of $P$ and $P'$, and hence, the two polyhedra have the same diameter.
Thus, \cite{dadush2016shadow} gives a bound $O((n-m)^3\log(n/\delta_C)/\delta_C)$ on the diameter of $P'$.
By the choice of $B$, from Lemma~\ref{lem:delta-kappa}\ref{i:m-kappa}, we obtain the diameter bound
 $O((n-m)^{3}m\kappa_{C^\top}$ $\log(\kappa_{C^\top}+n))$.
We claim that $\kappa_{C^\top}=\kappa_A$. Indeed, the kernels of $A=(A'|I_m)$ and $C^\top=(-I_{n-m}|(A')^\top)$ represent orthogonal complements, thus $\kappa_{C^\top}=\kappa_A$ by Proposition~\ref{prop:kappa-dual}. This completes the proof.
\end{proof}

The diameter bound in \cite{dadush2016shadow} is proved constructively, using the shadow simplex method. However, in the proof we choose $B$ maximizing $|\det(A_B)|$, a hard computational problem to solve even approximately \cite{DiSumma2014}. 
However, we do not actually require a (near) maximizing subdeterminant. For the argument, we only need to find a basis $B\subseteq [n]$ such that for $\bar A=A_B^{-1}A$, $\|\bar A\|_\infty\le \mu$ for some constant $\mu>1$. 
Then,  \eqref{eq:delta-kappa-alpha}, gives $1/\delta_{\bar A^\top}\le m\mu\kappa_A$.

Such a basis $B$ corresponds to  approximate local subdeterminant maximization, and 
 can be found using the following simple algorithm proposed by Knuth \cite{Knuth1985}. As long as there is an entry $|A_{ij}|>\mu$, then swapping $i$ for $j$ increases $|\det(A_B)|$ by a factor $|A_{ij}|>\mu$. 
Using that $|\det(A_B)|\le (\kd_W)^m$ by Proposition~\ref{prop:B_inverse_General}, the algorithm terminates in $O(m\log(\kd_W/\mu))$ iterations.

We also note that $\delta_A$ was also studied  for lattice basis reduction by Seysen \cite{seysen1993}. A related quantity has been used to characterize Hoffman constants (introduced in Section~\ref{sec:hoffman}), see \cite{Guler1995,Klatte1995,Pena2020}.

\section{Optimizing circuit imbalances}
\label{subsec:def_chi_bar_star}
Recall that $\bD_n$ is the set of  $n\times n$ positive definite diagonal matrices. For  every $D\in \bD_n$, $AD$ represents a column rescaling. This is a natural symmetry in linear programming, and particularly relevant in the context of interior point algorithms, as discussed in Section~\ref{sec:LLS}.

 The condition number $\Le_{AD}$ may vastly differ from $\Le_{A}$. In terms of the subspace $W=\ker(A)$, this amounts to rescaling the subspace by $D^{-1}$; we denote this by $D^{-1}W$. It is natural to ask for the best possible value that can be achieved by rescaling:
\[
\Le_W^* = \inf\left\{ \Le_{DW}:\,  D \in \bD_n\right\}\, .
\]
In most algorithmic and polyhedral results in this paper, the $\Le_W$ dependence can be replaced by $\Le^*_W$ dependence. For example, the diameter bound in Theorem~\ref{thm:kappa-diameter} is true in the stronger form with $\Le^*_W$, since the diagonal rescaling maintains the geometry of the polyhedron.

Even the value $\Le^*_W$ can be arbitrarily large.
As an example, let $W\subseteq \R^4$ be defined as $W=\spn\{(0,1,1,M),(1,0,M,1)\}$. Both generating vectors are elementary, and we see that $\Le_W=M$. Rescaling the third and fourth coordinates have opposite effect on the two elementary vectors, therefore we also have $\Le_{W}^*=M$, i.e. the original subspace is already optimally rescaled.

A key result in  \cite{DHNV20} shows that an approximately optimal rescaling can be found:
\begin{theorem}[\cite{DHNV20}]\label{thm:bar-chi-star}
There is an $O(n^2m^2 + n^3)$ time algorithm that for any matrix $A\in\R^{m\times n}$, computes an estimate $\xi$ of $\kappa_A$ such that
\[
\xi \leq \kappa_A \leq  (\kappa_A^*)^2 \xi
\]
and a $D\in \bD_n$ such that
\[
\kappa^*_A\le \kappa_{AD}\le (\kappa_A^*)^3\, .
\]
\end{theorem}
This is in surprising contrast with the inapproximability result Theorem~\ref{thm:chi-inapprox}. 
Note that there is no contradiction since the approximation factor $(\kappa_A^*)^2$ is not bounded as $2^{\mathrm{poly}(n)}$ in general.

The key idea of the proof of Theorem~\ref{thm:bar-chi-star} is to analyze the pairwise imbalances $\Le_{ij}=\Le^W_{ij}$ introduced in Section~\ref{sec:self-dual}.
In the 4-dimensional example above, we have $\kappa_{34}=\kappa_{43}=M$. Let $D\in {\mathbf D}$ and let $d\in\R^n$ denote the diagonal elements; i.e. the rescaling multiplies the $i$-th coordinate of every $w\in W$ by $d_i$.
Then, we can see that $\Le^{DW}_{ij}=\Le_{ij}d_j/d_i$.
In particular, for any pair of variables $i$ and $j$, $\Le^{DW}_{ij}\Le^{DW}_{ji}=\Le_{ij}\Le_{ji}$. Consequently, we get a lower bound $\Le_{ij}\Le_{ji}\le(\Le^*_W)^2$.

Theorem~\ref{thm:bar-chi-star} is based on a 
combinatorial min-max characterization that extends this idea. For the rest of this section, let us assume that the matroid ${\cal M}(W)$ is non-separable. In case it is separable, we can obtain $\kappa^*_W$ by taking a maximum over the non-separable components. 

Let  $G=([n],E)$ be the complete directed graph on $n$ vertices with 
edge weights $\kappa_{ij}$. Since ${\cal M}(W)$ is assumed to be non-separable, Proposition~\ref{prop:conn} implies that $\kappa_{ij}>0$ for any $i,j\in [n]$.
We will refer to this weighted digraph as  the \emph{circuit ratio digraph}.

Let $H$ be a \emph{cycle} in $G$, that is, a sequence of indices
$i_1,i_2,\dots,i_k, i_{k+1} = i_1$. We use $|H|=k$ to denote the
length of the cycle. (In this terminology, \emph{cycles}  refer to objects in $G$, whereas
\emph{circuits} to objects in $\circuits_W$.)

We use the notation $\Le(H)=\Le_W(H)=\prod_{j=1}^k \Le^W_{i_j
  i_{j+1}}$. The observation for length-2 cycles remains valid in general: $\Le(H)$ is invariant under any rescaling. This leads to the lower bound $(\Le(H))^{1/|H|}\le \kappa^*_W$. The
 best of these bounds turns out to be tight:
\begin{theorem}[\cite{DHNV20}]\label{thm:circuit-min-max}
For a subspace $W\subseteq \R^n$, we have
\[
\Le_W^* = \max\left\{\Le_W(H)^{1/|H|}:\ \mbox{$H$ is a cycle in $G$}\right\}\, .
\]
\end{theorem}
The proof relies on the following formulation:
\begin{align}\label{eq:min-mean-cycle-mult}
\begin{aligned}
    \Le^*_W=&&\min \; & t \\
            &&\Le_{ij}d_j/d_i & \leq t \quad \forall (i,j) \in E, \\
            &&d &> 0.
\end{aligned}
\end{align}
Taking logarithms and substituting $z_i=\log d_i$, we can rewrite this problem equivalently as
\begin{equation}\label{eq:min-mean-cycle}
\begin{aligned}
\min \; & s \\
\log \Le_{ij} + z_j - z_i & \leq s \quad \forall (i,j) \in E, \\
z &\in \R^n.
\end{aligned}
\end{equation}
This is the dual of the minimum-mean cycle problem with weights $\log\Le_{ij}$, and can be solved in polynomial time (see e.g. \cite[Theorem 5.8]{amo}).

Whereas this formulation verifies Theorem~\ref{thm:circuit-min-max}, it does not give a polynomial-time algorithm to compute $\Le^*_W$. In fact, the values $\Le_{ij}$ are already NP-hard to approximate due to Theorem~\ref{thm:chi-inapprox}. Nevertheless, the bound $\Le_{ij}\Le_{ji}\le(\Le^*_W)^2$ implies that for any elementary vector $g^C$ with support $i,j\in C$, we have 
\begin{equation}\label{eq:hat-sandwich}
\frac{\Le_{ij}}{(\Le^*_W)^2}\le \frac{1}{\Le_{ji}} \le \left|\frac{g_j}{g_i}\right|\le \Le_{ij}\, .
\end{equation}

To find an efficient algorithm as in Theorem~\ref{thm:bar-chi-star},
we  replace the exact values $\Le_{ij}$ by estimates $\hat\Le_{ij}$ obtained as $|g^C_j/g^C_i|$ for an arbitrary circuit $C\in\circuits_W$ with $i,j\in C$; these can be obtained using standard techniques from linear algebra and matroid theory. 
Thus, we can return $\xi=\max_{(i,j)\in E}\hat\Le_{ij}$ as the estimate on the value of $\Le_A$. To estimate $\kappa^*_A$, we solve \eqref{eq:min-mean-cycle} with the estimates $\hat\Le_{ij}$ in place of the $\Le_{ij}$'s. 

\subsection{Perfect balancing: $\Le^*_W=1$}
Let us now show that  $\kappa^*_A=1$ can be efficiently checked.
\begin{theorem}\label{thm:kappa-star-1}
There exists a strongly polynomial algorithm, that given a matrix $A\in\R^{m\times n}$, returns one of the following outcomes:
\begin{enumerate}[(a)]
\item \label{i:outcome_TU} 
A diagonal matrix $D\in{\bD}_n$ such that $\kappa_{AD}=1$ showing that $\kappa^*_A=1$. The algorithm also returns the exact value of $\kappa_A$.
Further, if $\ker(A)$ is a rational linear space, then we can select $D$ with integer diagonal entries that divide $\kd_A$.
\item \label{i:outcome_not_TU} The answer $\kappa^*_A>1$, along with a cycle of circuits $H$ such that $\kappa_A(H) > 1$.
\end{enumerate}
\end{theorem}
\begin{proof}
As noted above, we can assume without loss of generality that the matroid ${\cal M}(W)$ is non-separable, as we can reduce the problem to solving on all connected components separately.

We obtain estimates $\hat\Le_{ij}$ for every edge $(i,j)$ of the circuit ratio graph  using a circuit $C\in\circuits_W$ with $i,j\in C$. Assuming that $\Le^*_W=1$, \eqref{eq:hat-sandwich} implies that $\hat\kappa_{ij}=\kappa_{ij}$ holds and the rescaling factors $d_i$ must satisfy 
\begin{equation}\label{eq:kappa-d}
\hat\kappa_{ij}d_j=d_i\quad \forall i,j\in [n]\, . 
\end{equation}
If this system is infeasible, then using the circuits that provided the estimates $\hat\kappa_{ij}$, we can 
obtain a cycle  $H$ such that $\kappa_A(H)>1$, that is, outcome \ref{i:outcome_not_TU}. Let us now assume that \eqref{eq:kappa-d} is feasible; then it has a unique solution $d$ up to scalar multiplication.
We define $D\in\bD_n$ with diagonal entries $D_{ii}=d_i$. 

Since ${\cal M}(W)$ is non-separable, we can conclude that $\kappa^*_A=1$ if and only if $\kappa_{AD}=1$. By Theorem~\ref{thm:tu_iff_kappa_1}, this holds if and only if $A'=A_B^{-1} AD$ is a TU-matrix for any basis $B$.

We run Seymour's algorithm \cite{Seymour1980} for $A'$. If it confirms that $A$ is TU (certified by a construction sequence), then we return outcome \ref{i:outcome_TU}. In this case, $|g^C_j/g^C_i|$ is the same for any circuit $C$ with $i,j\in C$; therefore $\Le_{ij}=\hat\Le_{ij}$, and we can return $\Le_A=\max_{(i,j)\in E}\hat\Le_{ij}$.

Otherwise, Seymour's algorithm finds a $k\times k$ submatrix $T$ of $A'$ with $\det(T)\notin\{0,\pm1\}$. As in the proof of Proposition~\ref{prop:kappa-delta}, we can recover a circuit $C$ in $\circuits_{A'}=\circuits_{AD}$ with two entries $i,j\in C$ such that $|\bar g_j|\neq |\bar g_i|$ for the corresponding elementary vector $\bar g\in {\EE}(AD)$. Note that $\hat \Le_{ij} d_j=d_i$ for the rescaled estimates. Hence, the circuit $C'$ with $i,j\in C'$ used to obtain the estimate $\Le_{ij}$, together with $C$ certifies that $\Le^*_A>1$ as required for outcome \ref{i:outcome_not_TU}.

Finally, if $\ker(A)$ is a rational linear space and we concluded $\Le^*_A=1$, then let us select the solution $d_i$ to \eqref{eq:kappa-d} such that $d\in \Z^n$ and $\gcd(d)=1$. We claim that $d_i\divides\dot\kappa_W$ for all $i$. Indeed, let $k=\lcm(d)$. For each pair $i,j\in [n]$, $d_j/d_i = r/q$ for two integers $r,q\divides \kd_A$.
Hence, for any prime $p\in\primes$, $\nu_p(k)\le \nu_p(\kd_A)$, implying $k\divides\kd_A$. 
\end{proof}

Let $W\subseteq \R^n$ be a linear space such that ${\cal M}(W)$ is non-separable.
Recall from Theorem~\ref{thm:imbalance_triangle_inequality} that 
$\KK_{ij}^W\subseteq \KK_{ik}^W \cdot \KK_{kj}^W$ for all $i,j,k\in [n]$. We now characterize when equality holds for all triples.
\begin{prop}
\label{prop:reverse-kappa-1}
Let $W\subseteq \R^n$ be a linear space such that ${\cal M}(W)$ is non-separable. Then, the following are equivalent:
\begin{enumerate}[(i)]
\item\label{i:TU_equiv_i} $\Le^*_W=1$,
\item\label{i:TU_equiv_ii} $|\KK_{ij}^W|=1$ for all $i,j\in [n]$, 
\item\label{i:TU_equiv_iii}  $\KK_{ij}^W= \KK_{ik}^W \cdot \KK_{kj}^W$ holds for all distinct $i,j,k\in [n]$.
\end{enumerate}
\end{prop}
\begin{proof} 
\ref{i:TU_equiv_i} $\Leftrightarrow$ \ref{i:TU_equiv_ii}:
Consider any  rescaling  $D\in\bD_n$ with diagonal entries $d_i=D_{ii}$.
Then,  $\KK_{ij}^{DW}=\{1\}$, and $\KK_{ij}^{DW}=\frac{d_j}{d_i}\KK_{ij}^{W}$. Hence, if $\Le_{DW}=1$ for some $D\in\bD_n$, then  $\KK_{ij}^{DW}=\{1\}$ implying $|\KK_{ij}^W|=1$  for every $i,j\in [n]$. If $|\KK_{ij}^{W}|>1$ for some $i,j$, it follows that $\Le_{DW}\neq 1$ for any diagonal rescaling. 

\noindent \ref{i:TU_equiv_ii} $\Rightarrow$ \ref{i:TU_equiv_iii}: 
 We have $\KK_{ij}^W\subseteq \KK_{ik}^W \cdot \KK_{kj}^W$ by Theorem~\ref{thm:imbalance_triangle_inequality}. If all three sets are of size one, then equality must hold.

\noindent \ref{i:TU_equiv_iii} $\Rightarrow$ \ref{i:TU_equiv_i}: 
Let $i,j \in [n]$ arbitrary but distinct and let us define
 $$\Gamma_{ij} := \{\kappa(H) |\, \text{$H$ closed walk in $G$, $(i,j) \in E(H)$}\}\, .$$
 Note that either $\Gamma_{ij} = \{1\}$ or $\Gamma_{ij}$ is infinite as any cycle $H$ can be traversed multiple times to form a closed walk. Note that by (iii) we have for any $i,j \in [n]$ that
\begin{equation}
  \Gamma_{ij} \subseteq \bigcup  \Big\{\Pi_{(k,\ell) \in E(H)} \mathcal K_{k,\ell} |\, \text{$H$ closed walk in $G$, $(i,j) \in E(H)$}\Big\}
  = \mathcal K_{ij} \cdot\mathcal K_{ji}. 
\end{equation}
The set $\mathcal K_{ij} \cdot\mathcal K_{ji}$ is finite, implying that $\Gamma_{ij} = \{1\}$. This, together with Theorem~\ref{thm:circuit-min-max} gives \ref{i:TU_equiv_i}.
\end{proof}

A surprising finding by Lee \cite{Lee89,Lee90} is that if $\dot\kappa_W$ is an odd prime power, then $\kappa^*_W=1$ holds.\footnote{The statement in the paper is slightly more general, for $k$-adic subspaces with $k>2$; the proof is essentially the same.}
We first present a proof sketch following the lines of the one in \cite{Lee89,Lee90}.
We also present a second, almost self-contained proof,  relying only on basic results on TU matrices. \begin{theorem}[{Lee~\cite{Lee89,Lee90}}]\label{thm:kappa-star}
Each $W$ for which $\dot{\kappa}_W = p^\alpha$ where $p\in\primes$, $p>2$,  $\alpha\in \N$, then  $\kappa^*_W = 1$.
\end{theorem}

\begin{proof}
A theorem by Tutte \cite{Tutte1965} asserts that $W$ can be represented as the kernel of a unimodular matrix, i.e.\ $\kappa^*_W = 1$ or $W$ has a minor $W'$ such that $\circuits(W') \cong \circuits(U_2^4)$ where $U_2^4$ is the uniform matroid on four elements such that the independent sets are the sets of cardinality at most two. Here, a matroid minor corresponds to iteratively either deleting variables or projecting variables out. In the first case we are done, so let us consider the second case. Note that $W' \subset \R^4$ and by Lemma~\ref{lem:circuit_minors} we have that for all $i,j \in [4]$ we have that $\mathcal K_{ij}^{W'} \subseteq \mathcal K_{ij}^W$ and so in particular $\dot \kappa_{W'} = p^\beta$ for some $\beta \le \alpha$. An easy consequence of the proof of Proposition~\ref{prop:B_inverse} and the congruence $\circuits(W') \cong \circuits(U_2^4)$ is that $W'$ can be represented by $A'$, i.e.\ $\ker(A') = W'$ such that 
\begin{equation}
  A' = 
  \begin{bmatrix}
    1 & 0 & p^{\gamma_1} & p^{\gamma_2} \\
    0 & 1 & p^{\gamma_3} & p^{\gamma_4}
  \end{bmatrix}
\end{equation}  
for $\gamma_i \in \N \cup \{0\}$ and $i \in [4]$. Further, by $\circuits(W') \cong \circuits(U_2^4)$ and $\Delta_{A'} \divides \dot \kappa_{W'}$ (Proposition~\ref{prop:B_inverse}) we have that 
\begin{equation}
  \label{eq:p_alpha_contradiction}
  0 \neq \det\begin{pmatrix}
     p^{\gamma_1} & p^{\gamma_2} \\
     p^{\gamma_3} & p^{\gamma_4}
  \end{pmatrix} = p^{\gamma_1 + \gamma_4} - p^{\gamma_2 + \gamma_3} \divides p^\beta.
\end{equation}
It is immediate that \eqref{eq:p_alpha_contradiction} cannot be fulfilled for $p > 2$. 
\end{proof}

\begin{proof}[Alternative Proof of Theorem~\ref{thm:kappa-star}]
  Let $A \in \R^{m \times n}$ be such that $A = \ker(W)$ satisfying the properties in Proposition~\ref{prop:B_inverse} in basis form $A=(I_m|A')$; for simplicity, assume the identity matrix is in the first $m$ columns.
Let $G = \big([n], E(G)\big)$ be a directed multigraph associated with $A$ with edge set $E(G) = \bigcup_{k \in [m]} E_k(G)$ where $E_k(G) = \{(i,j) : A_{ki}A_{kj} \neq 0\}$. Further, define $\gamma\colon E(G) \to \R_+$ where for $e \in E_k$ we let $\gamma(e) = |A_{kj}/A_{ki}|$. For a directed cycle $C$ in $G$ we define $\gamma(C) := \prod_{e \in E(C)} \gamma(e)$.
\begin{claim}
All cycles $C$ in $G$ fulfill $\gamma(C) = 1$.
\end{claim}
\begin{proof}
For a contradiction, assume that there exists a cycle $C$ such that $\gamma(C) \neq 1$ and let $C$ be a shortest cycle with this property. Then $C$ has no chord $f \in E(G)$, as otherwise $C \cup \{f\}$ contains two shorter cycles $C_1, C_2$ such that $\gamma(C_1) \gamma(C_2) = \gamma(C) \neq 1$ and so in particular $\gamma(C_1) \neq 1$ or $\gamma(C_2) \neq 1$. This also means that the support of the corresponding submatrix $A_{I,J}$ of $A$ where $I := \{i \in [m] : E_i(G) \cap E(C) \neq 0\}$ and $J := V(C)$ is exactly the set of non-zeros of an incidence matrix of a cycle. We have that $\det(A_{I,J}) \neq 0$ as the corresponding cycle $C$ has $\gamma(C) \neq 1$. 
Recall the Leibniz determinant formula. As $A_{I,J}$ is supported on the incidence matrix of a cycle there exist only two bijective maps $\phi, \psi: I \to J$, $\phi \neq \psi$ such that $\prod_{i\in I}A_{i,\phi(i)} \neq 0 \neq \prod_{i \in I} A_{i, \psi(i)}$ is non-vanishing. One of the maps corresponds to traversing the cycle forward, the other corresponds to traversing it backwards. As all the entries of $A$ are powers of $p$ we therefore have that $0 \neq \det(A_{I,J}) = \pm p^\alpha \pm p^\beta $ for some $\alpha, \beta \in \N$. This contradicts Proposition~\ref{prop:B_inverse}\ref{part:subdet-p} for $p > 2$.
\end{proof}

The above claim implies the existence of a rescaling of rows and columns $\tilde A:= LAR$ where $L\in \bD_n$, $R \in \bD_m$ such that $\tilde A \in \{-1,0,1\}^{m \times n}$. If $\tilde A$ is TU, then we are done by Proposition~\ref{prop:kappa-delta} as now $\kappa_W^* =1$. Otherwise, we use a result by Gomory (see \cite{Camion1965} and \cite[Theorem 19.3]{SchrijverLPIP}) that states that any matrix $B$ with entries in $\{-1,0,1\}$ that is not totally unimodular has a submatrix $B'$ with $|\det(B')| = 2$. 
Let $I \subseteq [m]$ and $J \subseteq [n]$ such that $|\det(\tilde A_{I,J})| = 2$. Note that w.l.o.g.\ the diagonal entries of $L$ and $R$ are of the form $p^\alpha$ for some $\alpha \in \Z$. Therefore, $|\det(A_{I,J})| = \prod_{i \in I}L_{ii} \prod_{j \in J} R_{jj} |\det (\tilde A_{I,J})| = 2p^\beta$ for some $\beta \in \Z$. As $|\det(A_{I,J})| \in \N$ we must have $\beta \ge 0$ and  $2 \divides |\det(A_{I,J})|$. This again contradicts Proposition~\ref{prop:B_inverse}\ref{part:subdet-p} for $p > 2$.
\end{proof} 
\section{Hoffman proximity theorems}
\label{sec:hoffman}

Hoffman's seminal work \cite{Hoffman52} has analyzed proximity of LP solutions. Given $P=\{x\in \R^n:\, Ax\le b\}$, $x_0\in \R^n$, and norms $\|.\|_\alpha$ and $\|.\|_\beta$, we are interested in the minimum of 
$\|x-x_0\|_\alpha$ over $x\in P$.
Hoffman showed that this can be bounded as $H_{\alpha,\beta}(A)\|(Ax_0-b)^+\|_\beta$, where the Lipschitz-bound $H_{\alpha,\beta}(A)$ is a constant that only depends on $A$ and the norms. Such results are known as \emph{Hoffman proximity bounds} in the literature and have been extensively studied;  we refer the reader to \cite{Guler1995,Klatte1995,Pena2020} for references. In particular, they are related to  $\delta$ studied in Section~\ref{sec:delta}, see e.g. 
\cite{Guler1995}. 

In this section, we show a Hoffman-bound $H_{\infty,1}=\kappa_W$ for the system $x\in W+d,\, x\ge 0$, namely. Related bounds using $\bar\chi_A$ have been shown in \cite{ho2002}. We then extend it to proximity results on optimal LP solutions. These will be used in the black-box LP algorithms in Section~\ref{sec:black-box}, as well as for the improved analysis of the steepest descent circuit augmentation algorithm in Section~\ref{sec:steepest-improved}.

A central tool in this section are conformal decompositions into circuits as in Lemma~\ref{lem:sign-cons}. The next proof is similar to that of Proposition~\ref{lem:kappa-lift}.

\begin{lemma}\label{lem:feas}
 If the system $x\in W+d,\, x\ge 0$ 
   is feasible, then the system
  \begin{align*} 
  x&\in W+d\\
  \|x-d\|_\infty&\le \kappa_W \|d^-\|_1\\
  x&\ge 0,
  \end{align*}
  is also feasible.
  \end{lemma}
\begin{proof} 
Let $x$ be a solution to \ref{LP-subspace-f} such that $\|x-d\|_\infty$ is minimal, and subject to that, $\|x-d\|_1$ is minimal. Let $D=\{i\in [n]:\, d_i<0=x_i\}$.

  Take a conformal circuit decomposition of the  vector $x - d \in W$ as in Lemma~\ref{lem:sign-cons} in the form $x-d = \sum_{k = 1}^t g^k$ for some $t \in [n]$. We claim that $\supp(g^k)\cap D\neq \emptyset$ for all $k\in [t]$. Indeed, if $\supp(g^k)\cap D=\emptyset$, then $x'=x-\varepsilon g^k$ for some $\varepsilon>0$ is another solution with $\|x'-d\|_\infty\le \|x-d\|_\infty$ and $\|x'-d\|_1< \|x-d\|_1$.

Consider any index $i\notin D$. For every elementary vector $g^k$ with $i\in \supp(g^k)$, there exists an index $j\in D$ such that $|g^k_i|\le \Le_W |g_j^k|$. By conformity, 
\[
|x_i-d_i|=\left|\sum_{k=1}^t g^k_i\right|=\sum_{k=1}^t \left|g^k_i\right|\le\Le_W \sum_{j\in D} \sum_{k=1}^t\left|g^k_j\right|=\Le_W\sum_{j\in D} |d_j|\le \Le_W\|d^-\|_1\, ,
\]
completing the proof.
\end{proof}

We note that  \cite{DadushNV20} also provides a strongly polynomial algorithm that, for a given $z\in W+d$, and an estimate $\hat \Le$ on $\Le_W$, either finds a solution as in Lemma~\ref{lem:feas}, or finds an elementary vector that reveals $\hat \Le>\Le_W$.

We next provide proximity results for optimization.
For vectors $d,c\in \R^n$, let us define the set
\begin{equation}\label{theta-def}
\Lambda(d,c):=\supp(d^-)\cup \supp(c^+)\, .
\end{equation}
 Note that for $c\ge 0$,
$\left\|d_{\Lambda(d,c)}\right\|_1=\left\|d^-\right\|_1 +\left\|d_{\supp(c)}^+\right\|_1$.
Consequently, if $c\ge 0$ and $d_{\Lambda(d,c)}=0$, then $x=d$ and $s=c$ are optimal primal and dual solutions to \ref{LP-subspace-f}.

\begin{lemma}\label{lem:opt}
  If the system \ref{LP-subspace-f} is feasible, bounded and $c\ge 0$, then there is an optimal solution such that 
  \[
  \left\|x-d\right\|_\infty\le \Le_W \left\|d_{\Lambda(d,c)}\right\|_1. \]
  \end{lemma}
  \begin{proof}
Similarly to the proof of Lemma~\ref{lem:feas}, let $x$ be an optimal solution to \ref{LP-subspace-f} chosen such  that $\|x-d\|_\infty$ is minimal, and subject to that, $\|x-d\|_1$ is minimal; let $D=\{i\in [n]:\, d_i<0=x_i\}$. Take a conformal circuit decomposition $x-d = \sum_{i = 1}^t g^k$ for some $t \in [n]$. 
    With a similar argument as in the proof of Lemma~\ref{lem:feas}, we can show that for each $g^k$, either $\supp(g^k)\cap D\neq\emptyset$, or $g^k$ is an objective-reducing circuit, i.e. $\pr{g^k}{c} < 0$. Since $c\ge 0$, the latter requires that for some $i\in\supp(g^k)$, $g^k_i<0$ and $c_i>0$, implying $i\in d_{\supp(c)}^+$. A similar bound as in the proof of Lemma~\ref{lem:feas} completes the proof.
      \end{proof}

\begin{lemma}\label{lem:primal_prox}
  Let $W\subseteq \R^n$ be a subspace and $c, d \in \R^n$.
  Let $(\tilde x, s)$ be an optimal solution to
   LP$(W,\tilde x,c)$. Then there exists an optimal solution $(x^*, s^*)$ to LP$(W,d,c)$ such that
  \[
      \|x^* - \tilde x\|_\infty \le (\kappa_W + 1)\left\|\proj_{W^\perp}(d-\tilde x)\right\|_1 \, .
   \]
  \end{lemma}
  \begin{proof}
  Let $x=\tilde x + \proj_{W^\perp}(d - \tilde x)$. Note that $W+x=W+d$, and also $W^\perp+s=W^\perp+c$. Thus, the systems
    LP$(W,d,c)$ and LP$(W,x,s)$ define the same problem.
  
    We apply Lemma~\ref{lem:opt} to $(W,x,s)$. This guarantees the existence of an optimal $(x^*,s^*)$
   to  LP$(W,x,s)$  such that 
   \[
   \|x^*-x\|_\infty\le \kappa_W\|x_{\Lambda(x,s)}\|_1=\kappa_W\left( \left\|x^-\right\|_1+\left\|x_{\supp(s)}^+\right\|_1\right)\ .
   \]
   
   Since $\tilde x\ge 0$, we get that $\|x^-\|_1\le \|x_{\supp(x^-)}-\tilde x_{\supp(x^-)}\|_1$. Second, by the optimality of $(\tilde x, s)$, we have $\tilde x_{\supp(s^+)}=0$, and thus $x_{\supp(s^+)} = x_{\supp(s^+)}-\tilde x_{\supp(s^+)}$. These together imply that
   \[
   \begin{aligned}
   \|x^* - \tilde x\|_\infty &\le \|x^* - x\|_\infty + \|x - \tilde x\|_\infty \le (\kappa_W+1) \|x - \tilde x\|_1\\
   &= (\kappa_W + 1)\|\proj_{W^\perp}(d-\tilde x)\|_1\, . 
   \end{aligned}
   \]
  \end{proof}

    We can immediately use Lemma~\ref{lem:primal_prox} to derive a conclusion on the support of the optimal dual solutions to LP$(W,d,c)$, using  the optimal solution to LP$(W,\tilde d,c)$. 
    \begin{theorem}\label{cor:dual-fix-weak}
        Let $W\subseteq \R^n$ be a subspace and $c, d \in \R^n$.
         Let $(\tilde x, s)$ be an optimal solution to
         LP$(W,\tilde x, c)$ and 
        \[
        R:=\{i\in [n]:\,  \tilde x_i> (\kappa_W + 1)\|\proj_{W^\perp}(\tilde x-d)\|_1\}\, .
        \]
         Then for every dual optimal solution $s^*$ to \ref{LP-subspace-f}, we have $s^*_R=0$.
    \end{theorem}
    \begin{proof}
        By Lemma~\ref{lem:primal_prox} there exists an optimal solution $(x', s')$ to LP$(W,d,c)$ such that $\|x' - \tilde x\|_\infty \le (\kappa_W + 1)\|\proj_{W^\perp}(d - \tilde x)\|_1$.
        Consequently, $x'_R>0$, implying $s^*_R=0$ for every dual optimal $s^*$ by complementary slackness.
    \end{proof}

In Section~\ref{sec:steepest-improved}, we use a dual version of this theorem, also including upper bound constraints in the primal side. We now adapt the required proximity result to the following primal and dual LPs, and formulate it in matrix language to conform to the algorithm in Section~\ref{sec:steepest-improved}.
\begin{equation}
  \label{LP_upper}
  \begin{aligned}
\min \; & \pr{c}{x} \\ 
Ax&=b\,, \\
0\leq x&\leq u\,.\\
  \end{aligned}
  \quad\quad\quad
  \begin{aligned}
  \max \;  \pr{y}{b}-&\pr{u}{t} \\
  A^\top y + s-t &= c\,, \\
  s,t & \geq 0\,. \\
  \end{aligned}
\end{equation}
Note that any $y\in\R^m$ induces a feasible dual solution with $s_i=(c_i-\pr{a_i}{y})^+$ and $t_i=(\pr{a_i}{y}-c_i)^+$ for $i\in [n]$.
A primal feasible solution  $x$  and $y\in\R^m$ are optimal solutions if and only if $\pr{a_i}{y}\le c_i$ if $x_i<u_i$ and $\pr{a_i}{y}\ge c_i$ if $x_i>0$.

        \begin{theorem}\label{thm:primal-fixing-upper}
        Let $(x',y')$ be optimal primal and dual solutions to \eqref{LP_upper} for input $(b,u,c')$, and $(x'',y'')$ for input $(b,u,c'')$. Let
        \[
        \begin{aligned}
                R_0&:=\{i\in [n]:\,  \pr{a_i}{y'}< c'_i- (\kappa_W + 1)\|c'-c''\|_1\}\, ,\\
        R_u&:=\{i\in [n]:\,  \pr{a_i}{y'}> c'_i+ (\kappa_W + 1)\|c'-c''\|_1\}\, .
        \end{aligned}
        \]
         Then $x''_i=0$ for every $i\in R_0$ and $x''_i=u_i$ for every $i\in R_u$.
    \end{theorem}
    \begin{proof}
    Let $\bar A=\begin{pmatrix}A & 0\\ I_n & I_n\end{pmatrix}$. It is easy to see that $\kappa_{\bar A}=\kappa_A$. 
Let $\bar d\in \R^{2n}$ such that $\bar A\bar d=\begin{pmatrix} b\\ u\end{pmatrix}$.
    With $\bar c=(c,0_n)$,
    the primal system can be equivalently written as 
    $\min \pr{\bar c}{\bar x}$, $\bar x\in \ker(\bar A)+\bar d$, $\bar x\ge 0$. The statement follows by Theorem~\ref{cor:dual-fix-weak} applied for $W=(\ker(\bar A))^\perp=\operatorname{im}(\bar A^\top)$.
    \end{proof}

\section{Linear programming with dependence on the constraint matrix only}\label{sec:only-A}

Recent years have seen tremendous progress in the development of more efficient 
LP algorithms using interior point methods, see e.g. \cite{CLS19,LS19,Brand2020,Brand2021} and references therein. These algorithms are \emph{weakly} polynomial, i.e., their running time depends on the encoding length of the input $(A,b,c)$ of \ref{LP_primal_dual}.

A fundamental open problem is the existence of a \emph{strongly} polynomial LP algorithm; this was listed by Smale as one of the key open problems  in mathematics for the 21st century \cite{Smale98}. The number of arithmetic operations of such an algorithm would be polynomial in the number $n$ of variables  and $m$ of constraints, but independent of the input length.

Towards this end, there is a line of work on developing algorithms with running time depending only on the constraint matrix $A$, while removing the  dependence on $b$ and $c$. This direction was pioneered by Tardos's 1985 paper \cite{Tardos86}, giving an algorithm for \ref{LP_primal_dual} with integral $A$ that has runtime $\poly(m,n,\log \Delta_A)$.

A breakthrough work by  Vavasis and Ye \cite{Vavasis1996} introduced a \emph{Layered Least Squares Interior-Point Method} that solves \ref{LP_primal_dual} within $O(n^{3.5} \log (\bar \chi_A+n))$ iterations, each requiring to solve a linear system. Recall from Theorem~\ref{theorem:kappa_chi_bar_approx} that $\log (\bar \chi_A+n)=\Theta(\log(\kappa_A+n))$; also recall from  Proposition~\ref{prop:kappa-delta}  that $\kappa_A\le\Delta_A$.

Recently, \cite{DHNV20} improved the Vavasis--Ye bound to  $O(n^{2.5} \log(n) \log (\bar \chi_A^*+n))$ linear system solves, where $\bar \chi_A^*$ is the optimized version of $\bar\chi_A$, analogous to $\kappa^*_A$ defined in Section~\ref{subsec:def_chi_bar_star}. The key insight of this work is using  the circuit imbalance measure $\kappa_A$ as a proxy to $\bar\chi_A$. These results are discussed in Section~\ref{sec:LLS}.

Section~\ref{sec:black-box} exhibits another recent paper \cite{DadushNV20} that extends Tardos's black-box framework  to solve LP in runtime $\poly(m,n,\log (\bar \Le_A+n))$, based on the proximity results in Section~\ref{sec:hoffman}. We note that using an initial rescaling as in Theorem~\ref{thm:bar-chi-star}, we can obtain $\poly(m,n,\log (\bar \Le^*_A+n))$ runtimes from these algorithms.

\subsection{A black box algorithm}\label{sec:black-box}
The LP feasibility and optimization algorithms in \cite{DadushNV20} rely on a black-box subroutine for approximate LP solutions, and use their outputs to find exact primal (and dual) optimal solutions in time $\poly(m,n,\log (\bar \Le_A+n))$. For the black-box, one can use the fast interior-point algorithms cited above.

More precisely, we require  the following approximately feasible and optimal solution $\tilde x$ to \ref{LP-subspace-f} in time $\mathrm{poly}(n,m)\log((\bar\kappa_A + n)/\varepsilon)\big)$.
Here $\operatorname{OPT}_{\mathrm{LP}}$ denotes the objective value of \ref{LP-subspace-f}.
\begin{equation}
  \label{sys:near_feas_near_opt}
  \tag{\texttt{APX-LP}}
  \begin{aligned}
  \pr{c}{\tilde x} &\le \operatorname{OPT}_{\mathrm{LP}}+ \varepsilon \|c\|\cdot\|d\| \\ 
  \tilde x &\in W + d \\
  \|\tilde x^-\| &\le \varepsilon \|d\|\,\\
  \tilde x &\in \R^n 
  \end{aligned}
\end{equation}
The feasibility algorithm makes $O(m)$ calls, and the optimization algorithm makes $O(nm)$ calls to such a subroutine for $\varepsilon=1/(\bar\kappa_A + n)^{O(1)}$.

We now give a high-level outline of the feasibility algorithm in \cite{DadushNV20} for the system $x\in W+d$, $x\ge 0$. For the description, let us assume this system is feasible; in case of infeasibility, the algorithm recovers a Farkas-certificate.
The main progress step in the algorithm is reducing the dimension $m$ of the linear space $W$ by one,
based on information obtained form  an approximate solution  $\tilde x$  to \eqref{sys:near_feas_near_opt}. We can make such an inference using the proximity result Lemma~\ref{lem:feas}. 

If $\tilde x\ge 0$ for the solution returned by the solver, we can terminate with $x=\tilde x$. 
Otherwise, let $I \subseteq [n]$ denote the set of `large' coordinates of $\tilde x$, i.e., where $\tilde x_i>\kappa_W\|\tilde x^-\|_1$. By Lemma~\ref{lem:feas}, there must exist a feasible solution $x\in W+d$, $x\ge 0$ such that $x_i$ is still sufficiently large for $i\in I$.
Therefore,
 one can drop the sign-constraint on $I$, as non-negativity can be enforced automatically. We recurse on $\pi_J(W)$ for $J=[n]\setminus I$, i.e. project out the variables in $I$.

Each recursive call decreases the dimension of the subspace, until a feasible vector is found. The feasible solution on the remaining variables now has to be \emph{lifted} to the variables we projected out, to get a feasible solution to the original problem $x\in W + d$, $x\ge 0$. We use the \emph{lifting operator} $L_J^W(p)$ introduced in Section~\ref{subsec:lifting_operator}: for $p\in \pi_J(W)$, $z=L_J^W(p)$ is the minimum-norm vector in $W$ such that $z_J=p$.
According to Proposition~\ref{lem:kappa-lift}, $\|z\|_\infty\le \kappa_W\|p\|_1$; this bound can be used to guarantee that the lifted solution is nonnegative on $I$.

Algorithm~\ref{alg:feas} gives a simplified description of the 
feasibility algorithm of \cite{DadushNV20}. For simplicity, we ignore the infeasibility case and the details of the $\texttt{Adjust}(d)$ that may replace $d$ by it projection to $W^\perp$ in certain cases. This is needed to ensure $I\neq\emptyset$. Further, we omit an additional proximity condition from the approximate system \ref{sys:near_feas_near_opt}.

\begin{algorithm}[htb!]
  \caption{\texttt{Feasibility-Simplified}}
  \label{alg:feas}
  \SetKwInOut{Input}{Input}
  \SetKwInOut{Output}{Output}
  \SetKwComment{Comment}{$\triangleright$\ }{}
  \SetKw{And}{\textbf{and}}
  \Input{Instance of \ref{LP-subspace-f}, $\eps > 0$, with $c = 0_n$.}
  \Output{Feasible solution to the system in Lemma~\ref{lem:feas}}
   $d \gets \texttt{Adjust}(d)$ \Comment*{Occasionally applies a projection}
    \lIf{$d \in W$}{
      \Return{ $0_n$}
    }
    $\tilde x \gets \ref{sys:near_feas_near_opt}(W,d,0_n, \eps)$  \Comment*{Provided by Black box solver}
    $I \gets \{i : \tilde x_i \ge \kappa_W \|\tilde x^-\| \}, J \gets [n] \setminus I$,\;
    $z \gets \texttt{Feasibility-Simplified}(\pi_J(W), d_J, \eps)$\;
    \Return{$\begin{bmatrix} \tilde x_I + [L_J^W(z - \tilde x_J)]_I \\ z \end{bmatrix}$} \label{line:return}
\end{algorithm}

As stated here, the computation complexity is dominated by the at most $m$ recursive calls to the solver for the system \eqref{sys:near_feas_near_opt}.

In  \cite{DadushNV20}, these techniques are also extended to solve the minimum-cost problem \ref{LP-subspace-f} for arbitrary cost vector $c \in \R^n$. The  optimization algorithm  makes $O(m)$ recursive calls to the approximate solver to identify a variable $x_i$ that must be 0 in every optimal solution; this is deduced 
 using Theorem~\ref{cor:dual-fix-weak}.

\subsection{Layered least squares interior point methods}\label{sec:LLS}

In this section, we briefly review layered least squares (LLS) interior-point methods, and highlight the role of the circuit imbalance measure $\Le_W$ in this context.
The \emph{central path} for the standard log-barrier function is the  parametrized curve given by the solutions to the system following system for $\mu\in \R_{++}$
\begin{equation}
  \label{sys:central_path}
  \begin{aligned}
  Ax &= b \\
  A^\top y + s & = c\\
  x_is_i &= \mu \;\quad \forall i \in [n]\\
  x, s &>0
  \end{aligned}
\end{equation}
 A unique solution for each $\mu>0$ exist whenever \ref{LP_primal_dual} possesses strictly feasible primal and dual solutions, i.e.\ primal resp.~dual solutions with $x > 0$ resp.~$s > 0$; the duality gap between these solutions is $n\mu$. The limit point at $\mu\to 0$ gives a pair of primal and dual optimal solutions.
At a high level,  interior point methods require an initial solution close to the central path for some large $\mu$ and proceed by following the central path in some proximity towards smaller and smaller $\mu$, which corresponds to converging to an optimal solution. A standard variant is the Mizuno--Todd--Ye \cite{MTY} \emph{predictor-corrector} method. This alternates between \emph{predictor} and \emph{corrector} steps. Each predictor step decreases the parameter $\mu$ at least by a factor $(1-\beta/\sqrt{n})$, but moves further away from the central path. Corrector steps maintain the same $\mu$ but restore better centrality. 

Let us now focus on the predictor step at a given point $(x,s)$; we use the subspace notation $W=\ker(A)$, $W^\top=\operatorname{im}(A^\top)$ as in \ref{LP-subspace-f}. The augmentation direction is computed by the \emph{affine scaling (AS) step}, that can be written as 
weighted least squares problems on the primal and dual sides:
\begin{equation}
  \label{sys:affine_scaling}
  \begin{aligned}
\Delta x &:= \argmin \Big\{\sum_{i \in [n]} \Big(\frac{x_i + \Delta x_i}{x_i}\Big)^2 : \Delta x \in W\Big\} \\
\Delta s &:= \argmin \Big\{\sum_{i \in [n]} \Big(\frac{s_i + \Delta s_i}{s_i}\Big)^2 : \Delta s \in W^\perp\Big\} \\
  \end{aligned}
\end{equation}
The update is then performed by setting $x \gets x + \alpha \Delta x$, $s \gets s + \alpha \Delta s$
for some $\alpha \in [0,1]$. As such, this algorithm can find $\varepsilon$-approximate solutions in weakly polynomial time. However,  it does not even terminate in finitely many iterations,
because using the weighted $\ell_2$-regressions problems \eqref{sys:affine_scaling} will never set variables exactly to 0
as required by complementary slackness. For standard interior point methods, a final rounding step is required.

The layered least squares interior-point method by Vavasis and Ye \cite{Vavasis1996} not only  terminates finitely, but has an iteration bound
of $O(n^{3.5}\log(\bar \chi_W + n))$, depending only on $A$, but independent of $b$ and $c$. We will refer to this as the VY algorithm.

Recall from Theorem~\ref{theorem:kappa_chi_bar_approx} that $\log(\bar \chi_A + n)=\Theta(\log(\kappa_A + n))$.
 For certain predictor iterations, they use a
 \emph{layered least squared (LLS)} step instead of affine scaling.
Variables are split into layers according to the $x_i$ values: we order the variables as $x_1\ge x_2\ge\ldots\ge x_n$, and start a new layer whenever there is a big gap between consecutive variables, i.e. $x_i>O(n^2)\bar\chi_A x_{i+1}$. For a point $(x,s)$ on the central path, the ordering on the $s_i$'s will be approximately reverse.

We illustrate their step based on a partition of the variable set $[n]$ into two layers $B \cup N = [n]$; the general step may use an arbitrary number of layers.
The layered least squared step is given in a 2-stage approach via
\begin{equation}
  \label{sys:lls}
  \begin{aligned}
\Delta x_N^\lal &:= \argmin \Big\{\sum_{i \in N} \Big(\frac{x_i + \Delta x_i}{x_i}\Big)^2 : \Delta x_N \in \pi_N(W)\Big\}, \\
\Delta x_B^\lal &:= \argmin \Big\{\sum_{i \in B} \Big(\frac{x_i + \Delta x_i}{x_i}\Big)^2 : (\Delta x_B, \Delta x_N)\in W\Big\}, \\[10pt]
\Delta s_B^\lal &:= \argmin \Big\{\sum_{i \in B} \Big(\frac{s_i + \Delta s_i}{s_i}\Big)^2 : \Delta s_B \in \pi_B(W^\perp)\Big\}, \\
\Delta s_N^\lal &:= \argmin \Big\{\sum_{i \in N} \Big(\frac{s_i + \Delta s_i}{s_i}\Big)^2 : (\Delta s_B, \Delta s_N)\in W^\perp\Big\}. \\
  \end{aligned}
\end{equation}
Whereas the  predictor-corrector algorithm---as most standard interior point variants---is invariant under rescaling the columns of the constraint matrix, the VY algorithm is \emph{not}: the layers are chosen by comparing the $x_i$ values. For this reason, it was long sought to find a scaling-invariant version of \cite{Vavasis1996}, that would automatically improve the running time dependence from $\bar\chi_W$ to the best possible value $\bar\chi^*_W$ achievable under column rescaling. In this line of work fall the results of \cite{Monteiro2008,Monteiro2003,MonteiroT05}, but none of them achieving dependence on the constraint matrix only while being scaling-invariant. 

This question was finally settled in \cite{DHNV20} in the affirmative. A key ingredient is revealing the connection between $\bar\chi_W$ and $\Le_W$. Preprocessing the instance via  the algorithm in Theorem~\ref{thm:bar-chi-star} to find a nearly optimal rescaling for $\Le_W$ (and thus for $\bar\chi_W$), and then using the VY algorithm already achieves $O(n^{3.5}\log(\bar \chi^*_W + n))$.
Beyond this, \cite{DHNV20} also presents a new LLS interior point method based on the pairwise circuit imbalances $\Le_{ij}$ that is inherently scaling invariant, as well as an improved analysis of $O(n^{2.5}\log(n)\log(\bar \chi^*_W + n))$ iterations.
 We give an outline next. Details are omitted here, a self-contained 
 overview  can be found in \cite{DHNV20}.

\paragraph{What determines a good layering?}
We illustrate the LLS step in the following hypothetical situation. Assume that the partition $B \cup N$ is such that $x_N^* = 0$ and $s_B^*=0$ for  optimal primal and dual solution $(x^*,s^*)$ to \ref{LP-subspace-f}. 
In particular, the LLS direction in \eqref{sys:lls} will set $\Delta x_N^\lal = -x_N$. Note that this does not hold for the AS direction $\Delta x_N$ that solves \eqref{sys:affine_scaling}. This benefit of the LLS direction over the AS direction will result in us being able to choose $\alpha$ larger for the LLS step compared to the AS step. 

To terminate with an optimal solution in a single step, we need to be able to select the step size $\alpha=1$, which requires that $x_B+\Delta_B^\lal\ge 0$.
But as in the computation of $\Delta x_N^\lal$ the components in $B$ are ignored we need to ensure the choice of $\Delta x_N^\lal$ does not impact $\Delta x_B^\lal$ by too much. By that we mean that there is a vector $z \in \R^B$ such that $|z_i/x_i| \ll 1$ for all $i \in B$ and $(z, \Delta x_N) \in W$. The norm of this $z$ is exactly governed by the lifting operator we introduced in Section~\ref{subsec:lifting_operator}. Let $W^x=\operatorname{diag}(x)^{-1}W= \{(w_i/x_i)_{i \in [n]} : w \in W\}$ denote the space $W$ rescaled by the $1/x_i$ values. Then,
\begin{equation}
\Big\|\Big(\frac{z_i}{x_i}\Big)_{i \in B}\Big\| = \Big\|L_N^{W^x}\Big(\Big(\frac{\Delta x_i^\lal}{x_i}\Big)_{i \in N}\Big)\Big\| \le \|L_N^{W^x}\| \cdot \Big\|\Big(\frac{\Delta x_i^\lal}{x_i}\Big)_{i \in N}\Big\|\, .
\end{equation}
By Lemma~\ref{lem:kappa-lift}, note that
\begin{equation}
  \Big\|\Big(\frac{z_i}{x_i}\Big)_{i \in B}\Big\| \le n\kappa_{W^x} \Big\|\Big(\frac{\Delta x_i}{x_i}\Big)_{i \in N}\Big\|\, .
\end{equation}
Further, notice that the lifting cost imposed on variables in $B$ by $\Delta x_N^\lal$ are given by the circuit imbalances in the rescaled space $W^x$: For $i \in B$ and $j \in N$ we are interested in $\kappa^{W^x}_{ji} = \kappa_{ji}{x_j}/{x_i}$.
In particular, if these quantities are small for all $i\in B$ and $j\in N$, then the \emph{low lifting cost} discussed above is achieved and we can select stepsize $\alpha=1$.

\paragraph{The choice of layers}
The Vavasis-Ye algorithm defines the layering based on the magnitude of the elements $x_i$. This guarantees that  $\kappa_{ji}{x_j}/{x_i}$ is small since 
$x_i>O(n^2)\kappa_A x_{j}$ if $i$ is on a higher layer than $j$. However, this choice is inherently not scaling-invariant.

The LLS algorithm in \cite{DHNV20} directly uses the scaling invariant quantities $\kappa_{ji}{x_j}/{x_i}$ to define the layering. In the ideal version of the algorithm, the layers are selected as the strongly connected components of the directed graph formed by the edges where this value is large. Hence, $\kappa_{ji}{x_j}/{x_i}$ is small whenever $i$ is on a higher layer than $j$.

This ideal version cannot be implemented since the pairwise imbalances $\kappa_{ji}$ are hard to compute or even approximate. The actual algorithm instead works with lower estimates $\hat\kappa_{ji}$. Thus, we may miss some edges from the directed graph, in which case the lifting may fail. Such failure will be detected in the algorithm, and in turn reveals better estimates for some pairs $(i,j)$.

\subsection{The curvature of the central path}
The condition number $\bar\chi_A^*$ also has an interesting connection to the geometry of the central path. In this context, Sonnevend, Stoer, and Zhao~\cite{Sonnevend1991} introduced a primal-dual curvature notion. 
Monteiro and
Tsuchiya~\cite{Monteiro2008} reveal strong connections between the curvature integral, the Mizuno-Todd-Ye predictor-corrector algorithm, and the Vavasis-Ye algorithm. In particular, they prove a bound 
$O(n^{3.5} \log(\bar{\chi}^*_A+n))$ on the curvature integral.

Besides the above primal-dual curvature, one can also study the \emph{total curvature} of the central path, a standard notion in algebraic geometry. De Loera, Sturmfels, and Vinzant \cite{Loera2012} studied the \emph{central curve} defined as the solution of the polynomial equations
\begin{equation}\label{eq:central-curve}
Ax=b\, ,\quad A^\top y-s=c\, \quad x_i s_i=\lambda\, \quad \forall i\in[n],\quad 
\end{equation}
This includes the usual central path in the region $x,s>0$, but also includes the central path of all other LPs with objective $c$ in the hyperplane arrangement in $\{x\in\R^n:\, Ax=b\}$ defined by the hyperplanes $x_i=0$; i.e., all LPs where some nonnegativity constraints $x_i\ge 0$ are flipped to $x_i\le 0$.
 In fact, \cite{Loera2012} shows thato \eqref{eq:central-curve} defines the smallest algebraic variety containing the central path.
 
 They consider the average curvature taken over the bounded regions in the hyperplane arrangement, and show a bound $2\pi(n-m-1)$ for the primal central path (i.e., the projection of \eqref{eq:central-curve} to the $x$ space), and $2\pi(m-1)$ for the dual central path (the projection to the $s$ space). Their argument crucially relies on circuit polynomials defined via elementary vectors.
  See \cite{Loera2012} for further pointers to the literature on the total curvature of the central path. 

\section{Circuit diameter bounds and circuit augmentation algorithms}\label{sec:augment}
Consider an LP in standard equality form with upper  bounds, where $A\in \R^{m\times n}$, $b\in \R^m$, $u\in \R^n$:
\begin{equation}
\label{LP-augment}\tag{LP$(A,b,c,u)$}
\begin{aligned}
\min \; & \pr{c}{x} \\ 
Ax&=b \\
0\leq x&\leq u\,\\
\end{aligned}
\end{equation}

In Section~\ref{sec:delta} we briefly mentioned the Hirsch conj\-ecture and some progress towards the polynomial Hirsch conjecture; in Theorem~\ref{thm:kappa-diameter} shows a bound $O((n-m)^3m\Le_A\log(\Le_A+n))$ on the diameter of $\{x\in \R^n:\, Ax=b,x\ge 0\}$ for $A\in\R^{m\times n}$.

\emph{Circuit diameter bounds} were introduced by Borgwardt, Finhold, and Hemmecke~\cite{Borgwardt2015} as a relaxation of diameter bounds. Let $P$ denote the feasible region of \ref{LP-augment}.
A \emph{circuit walk} is a set of consecutive feasible points $x^{(1)},x^{(2)},\ldots,x^{(k+1)}\in P$ such that for each $t=1,\ldots,k$, $x^{(t+1)}=x^{(t)}+g^{(t)}$ for $g^{(t)}\in \EE(A)$, and further, 
$x^{(t)}+(1+\varepsilon) g^{(t)}\notin P$ for any $\varepsilon>0$, i.e., each consecutive circuit step is \emph{maximal}. The \emph{circuit diameter} of $P$ is the minimum length of a circuit walk between any two vertices $x,y\in P$. 

In contrast to walks in the vertex-edge graph, circuit walks are non-reversible and the minimum length from $x$ to $y$ may be different from the one from $y$ to $x$; this is due to the maximality requirement. The circuit-analogue of the Hirsch conjecture, formulated in \cite{Borgwardt2015}, asserts that the circuit diameter of a polytope in $d$ dimensions with $n$ facets is at most $n-d$; this  may be true even for unbounded polyhedra, see \cite{Borgwardt2018circuit}.

In this section we begin by showing a recent, improved bound on the circuit diameter with $\log \kappa_A$ dependence. Section~\ref{subsec:circ-diameter-algorithms} gives an overview of  circuit augmentations algorithms. We review existing algorithms for different augmentation rules (Theorem~\ref{thm:DLHL} and Theorem~\ref{thm:opt-augment}), and also show a new bound for the steepest-descent direction (Theorem~\ref{thm:steepest-main}). The bounds in these three theorems also translate directly to circuit diameter bounds, since they all consider algorithms with maximal augmentation sequences.

\subsection{An improved circuit diameter bound}

In a recent paper Dadush et al. \cite{dadush2021circuit} gave the following bound on the circuit diameter:

\begin{theorem}[\cite{dadush2021circuit}]\label{thm:diam-augment} An LP of the form \ref{LP-augment} has circuit diameter bound
$O(m^2 \log(m + \kappa_A) + n \log n)$.
\end{theorem}

Let us highlight the main ideas to prove Theorem~\ref{thm:diam-augment}. The argument is constructive but non algorithmic in the sense that the augmentation steps are defined using the optimal solution.
We first show the bound in Theorem~\ref{thm:diam-augment} for \ref{LP_primal_dual} (i.e., without upper bounds $u$) and then extend the argument to systems of form \ref{LP-augment}. Let $x^*$ be a basic optimal solution to \ref{LP_primal_dual} corresponding to basis $B$, and let $N = [n]\setminus B$. Thus, $x^*$ is the unique optimal solution with respect to the cost vector $c=(0_B,\1_N)$.

For the current iterate $x^{(t)}$, consider a conformal circuit decomposition $h^1, \ldots, h^k$ of $x^* - x^{(t)}$, and select a circuit $h^i, i \in [k]$ such that $\|h_N^i\|_1$ is maximized.
We find the next iterate $x^{(t+1)}=x^{(t)}+\alpha h^i$ for the maximal stepsize $\alpha>0$.
 Note that the existence of such a decomposition \emph{does not} yield a circuit diameter bound $n$ due to the maximality requirement in the definition of circuit walks. 
Nonetheless, it can be shown that we will not overshoot $h^i$ by too much. More precisely, one can show that the step length will be $\alpha\in [1,n]$. Further, the choice of $h^i$ guarantees that $\|x^{(t)}_N\|_1$ decreases geometrically.

The analysis focuses on the index sets $L_t=\{i\in [n]:\, x^*_i>n\kappa_A\|x^{(t)}_N\|_1\}$ and $R_t=\{i\in [n]:\, x^{(t)}_i\le nx^*_i\}$. For every $i\in L_t$, $x_i$ must already be `large' and cannot be set to zero later in the algorithm; $R_t$ is the set of indices that have essentially `converged' to the final value $x^*_i$. Since $\|x^{(t)}_N\|_1$ is decreasing, once an index enters $L_t$, it can never leave again. The same property can be shown for $R_t$. Moreover, a new index is added to either set $R_t$ or $L_t$ every $O(m\log(m+\kappa_A))$ iterations, leading to the
 overall bound $O(m^2 \log(m + \kappa_A))$. 

For a system \ref{LP-augment} with upper bounds $u$, the above argument yields a bound $O(n^2\log(n + \kappa))$ using a simple reduction. To achieve a better bound, \cite{dadush2021circuit} gives a preprocessing sequence of $O(n \log n)$  circuit augmentations that reduces the number of variables to $\le 2m$.
This preprocessing terminates once the set of columns in $A_D$ are linearly independent for $D=\{i \in [n] : x_i^* \neq x_i \text{ and } x_i^* \in \{0, u_i\} \}$. Since a basic solution $x^*$ may have $\le m$ entries not equal to the lower or upper bound, at this point there are $\le 2m$ variables $x_i\neq x_i^*$. This leads to a circuit diameter bound of $O(m^2 \log(m + \kappa_A) + n \log n)$.

\subsection{Circuit augmentation algorithms}

\label{subsec:circ-diameter-algorithms}
The generic \emph{circuit augmentation algorithm} is a circuit walk  $x^{(1)},x^{(2)},\ldots,x^{(k+1)}\in P$ as defined above, such that an initial  feasible $x^{(0)}$ is given, 
and $\pr{c}{x^{(t+1)}}<\pr{c}{x^{(t)}}$, i.e., the objective value decreases in every iteration.
The elementary vector $g$ is an \emph{augmenting direction} for the solution $x^{(t)}$
 if and only if $\pr{c}{g}<0$, $g_i\ge 0$ for every $x^{(t)}_i=0$ and $g_i\le 0$ for every  $x^{(t)}_i=u_i$. By LP duality, $x^{(t)}$ is optimal if and only if no augmenting direction exists. Otherwise, the algorithm proceeds to the next iterate  $x^{(t+1)}$ by a maximal augmentation in an augmenting direction.

The simplex algorithm can be seen as a circuit augmentation algorithm that is restricted to using special elementary vectors corresponding to edges of the polyhedron.\footnote{Simplex may contain degenerate pivots when the basic solution remains the same; we do not count these as augmentation steps.} For the general framework, the iterates $x^{(k)}$ may not be vertices. However, in case of maximal augmentations, they must all lie on the boundary of the polyhedron.

\medskip

In unpublished work, Bland~\cite{Bland76} extended the Edmonds--Karp--Dinic algorithm \cite{dinic1970algorithm,Edmonds72} algorithm for general LP, see also \cite[Proposition 3.1]{Lee89}. Circuit augmentation algorithms were revisited by De Loera, Hemmecke, and Lee in 2015 \cite{DHL15}, analyzing different augmentation rules and extending them to integer programming. We give an overview of their results first for linear programming. In particular, they studied three augmentation rules that use maximal augmentation. Let  $x^{(t)}$ be the current feasible solution, and we aim to select an augmenting direction $g$ as follows.
\begin{itemize}
	\item \emph{Dantzig-descent direction:} Select $g$ such that $-\pr{c}{g}$ is maximized, where $g=g^C$ is the elementary vector with $\lcm(g^C)=1$ for a circuit $C\in \circuits_W$.
	\item \emph{Deepest-descent direction:} Select $g$ such that $-\alpha\pr{c}{g}$ is maximized, where $\alpha$ is the maximal stepsize for $x^{(t)}$ and $g$.
	\item \emph{Steepest-descent direction:} Select $g$ such that $-\pr{c}{g}/\|g\|_1$ is maximized.
\end{itemize}

Computing Dantzig- and deepest-descent directions is in general NP-hard, see \cite{DKS19} and  as detailed below. The steepest-descent direction can be formulated by an LP; but without any restrictions on the input problem, this may not be simpler than the original one. However, it could be easier to solve in practice; Borgwardt and Viss \cite{BorgwardtV2020} exhibits an implementation of a steepest-descent circuit augmentation algorithm with encouraging computational results.

\subsubsection{Augmenting directions for flow problems} 

 It is instructive to consider these algorithms  for the special case of \emph{minimum-cost flows}. Given a directed graph $D=(V,E)$ with capacities $u\in \R^E$, costs $c\in \R^E$, and node demands $b\in \R^V$ with $b(V)= \sum_{i \in V} b_i = 0$. The objective is to find the minimum cost flow $x$ that satisfies the capacity constraints: $0\le x\le u$, and the node demands: for each node $i\in V$, the total incoming minus the total outgoing flow equals $b_i$. This can be written in the form \ref{LP-augment} with $A$ as the node-arc incidence matrix of $D$, a TU matrix. 
Let us define the \emph{residual graph} $D_x=(V,E_x)$, where 
for $(i,j)\in E_x$ we let $(i,j)\in E$ if $x_{ij}<u_{ij}$ and $(j,i)\in E$ if $x_{ij}>0$. The cost of a reverse arc will be defined as $c_{ji}=-c_{ij}$. We will also refer to the \emph{residual capacities} of arcs; these are $u_{ij}-x_{ij}$ in the first case and $x_{ij}$ in the second. 

Let us observe that the augmenting directions correspond to directed cycles in the residual graph. Circuit augmentation algorithms for the primal and dual problems yield the rich classes of cycle cancelling and cut cancelling algorithms, see the survey \cite{Shigeno2000}.

The \emph{maximum flow problem} between a source $s$ and sink $t$ can be formulated as a special case as follows. We add a new arc $(t,s)$ with capacity $\infty$, set the demands $b\equiv 0$, and costs as $c_{ts}=-1$ and $c_{ij}=0$ otherwise. 
Bland's \cite{Bland76} observation was that the steepest-descent direction for this problem corresponds to finding a shortest residual $s$-$t$ path, as chosen in the Edmonds--Karp--Dinic algorithm. 

More generally, a steepest-descent direction amounts to finding a residual cycle $C\subseteq E_x$ that minimizes the mean cycle cost $c(C)/|C|$. Thus, the steepest descent algorithm for minimum-cost flows corresponds to the classical Goldberg--Tarjan algorithm \cite{Goldberg89} that is strongly polynomial with running  time $O(|V|\cdot |E|^2)$ \cite{Radzik-Goldberg}.

Let us now consider the other two variants. A Dantzig-descent direction in this context asks for the most negative cycle, i.e., a cycle maximizing $-c(C)$. A deepest-descent direction asks for a cycle $C$ of arcs that maximizes $-\alpha c(C)$, where $\alpha$ is the residual capacity of $C$. Computing both these directions exactly is NP-complete, since they generalize the  Hamiltonian-cycle problem: for every  directed graph, we can set up a flow problem where $E_x$ coincides with the input graph, all residual capacities are equal to 1, and all costs are $-1$. We note that De Loera, Kafer, and Sanit\`a~\cite{DKS19} showed that computing the Dantzig- and deepest-descent directions is also NP-hard for the fractional matching polytope.

Nevertheless, the deepest-descent direction can be suitably approximated.
Wallacher \cite{Wallacher} proposed selecting a \emph{minimum ratio cycle} in the residual graph. This is a cycle in $E_x$ that minimizes $c(C)/d(C)$, where $d_e=1/u_e$ for every residual arc $e\in E_x$; such a cycle can be found in strongly polynomial time. It is easy to show that this cycle approximates the deepest descent direction within a factor $|E_x|$.
Wallacher's  algorithm can be naturally extended to  linear programming \cite{mccormick-shioura-not-strongly}, and has found several combinatorial applications, e.g. \cite{Wallacher1999,Wayne02}, and has also been used in the context of integer programming \cite{Schulz1999}. We discuss an improved new variant in Section~\ref{sec:log-kappa}.
A different relaxation of the deepest-descent algorithm was given by Barahona and Tardos \cite{Barahona1989}, based on Weintraub's algorithm \cite{Weintraub1974}.

\subsubsection{Convergence bounds}
We now state the convergence bounds from \cite{DHL15}. The original statement refers to subdeterminant bounds; we paraphrase them in terms of finding approximately optimal solutions.

\begin{theorem}[{De Loera, Hemmecke, Lee \cite{DHL15}}]\label{thm:DLHL}
Consider a linear program in the form \ref{LP-augment}. 
Assume we are given an initial feasible solution $x^{(0)}$, and let $\OPT$ denote the optimum value. By an $\varepsilon$-
 optimal solution we mean an iterate $x^{(t)}$ such that $\pr{c}{x^{(t)}}\le \OPT+\varepsilon$.
\begin{enumerate}[(a)]
	\item \label{i:deepest}For given $\varepsilon>0$, one can find an $\varepsilon$-optimal solution in
	$2n\log_2\left(\frac{\pr{c}{x^{(0)}}-\OPT}{\varepsilon}\right)$ deepest-descent augmentations.
	\item\label{i:dantzig} For given $\varepsilon>0$, one can find an $\varepsilon$-optimal solution in
	$\frac{2n^2\gamma}{\varepsilon}\log_2\left(\frac{\pr{c}{x^{(0)}}-\OPT}{\varepsilon}\right)$  Dantzig-descent augmentations, where $\gamma$ is an upper bound on the maximum entry in any feasible solution.
	\item\label{i:steepest} One can find an exact optimal solution in $\min\{n|\circuits_A|,\ell_A\}$ steepest-descent augmentations, 
	 where $\ell_A$ denotes the number of distinct values of $\pr{c}{g}/\|g\|_1$ over $g\in \EE(A)$.
\end{enumerate}
\end{theorem}

In general, circuit augmentation algorithms may not even finitely terminate; see \cite{mccormick-shioura-not-strongly} for an example on Wallacher's rule for minimum cost flows.
In parts \ref{i:deepest} and \ref{i:dantzig}, assume that all basic solutions are $1/k$-integral for some $k\in \mathbb{Z}$ and cost function is $c\in \Z^n$. If $x^{(t)}$ is a $\varepsilon$-optimal solution for $\varepsilon<1/k$, then we can identify an optimal vertex of the face containing $x^{(t)}$ using 
a Carath\'eodory decomposition argument,
this can be implemented by a sequence of $\le n$ circuit augmentations (see \cite[Lemma 5]{DHL15}).

According to part \ref{i:steepest}, steepest descent terminates with  an optimal solution in a finite number of iterations; moreover,  the bound only depends on the linear space $\ker(A)$ and $c$, and not on the parameters $b$ and $u$. However, the bound can be exponentially large.

Bland's original observation was that $\ell_A$ is strongly polynomially bounded for the maximum flow problem. Recall that all elementary vectors $g$ correspond to cycles in the auxiliary graph. Normalizing such that $g_i\in \{0,\pm 1\}$,  $-\pr{c}{g}=1$ for every augmenting cycle (as these must use the $(t,s)$ arc), and $\|g\|_1$ is between 1 and $|E|$. In fact, the crucial argument by Edmonds and Karp \cite{Edmonds72} and Dinic \cite{dinic1970algorithm}  is showing that the length of the shortest augmenting path is non-decreasing, and must strictly increase within $|E|$ consecutive iterations.

For an integer cost function $c\in \Z^n$, Lee~\cite[Proposition 3.2]{Lee89}  gave the following upper bound on $\ell_A$:
\begin{prop}\label{prop:l-a} If $\|c\|_1\le (n-m+1)\|c\|_\infty$, then 
\[
\ell_A\le\frac12 {\|c\|_\infty} (n-m+1)\bar\Le_A((n-m+1)\bar\Le_A +1)\, .
\]
\end{prop}

 In order to bound the circuit distance between vertices $x$ and $y$ let us use the following cost function. For the basis $B$ defining $y$, let 
\begin{equation}
\label{eq:c-u}
c_i=\begin{cases}0&\mbox{if }i\in B\, ,\\
1& \mbox{if }i\in [n]\setminus B, y_i=0\, ,\\
-1&\mbox{if }i\in [n]\setminus B, y_i=u_i\, .
\end{cases}
\end{equation}
With this cost function, Theorem~\ref{thm:DLHL}\ref{i:steepest} and Proposition~\ref{prop:l-a} yield  a bound $O((n-m)^2 \bar\Le_A^2)$ on the circuit diameter using the steepest descent algorithm.

Extending the analysis of the Goldberg-Tarjan algorithm \cite{Goldberg89}, we present a new bound that only depends on the fractional circuit imbalance $\kappa_A$, and is independent of $c$.
The same bound was independently obtained by Gauthier and Derosiers \cite{Gauthier2021}. The proof is given in Section~\ref{sec:steepest-improved}.

\begin{theorem}\label{thm:steepest-main} For the problem \ref{LP-augment} with constraint matrix $A\in\R^{m\times n}$, the steepest-descent algorithm terminates within $O(n^2m\kappa_A \log(\kappa_A+n))$ augmentations starting from any feasible solution $x^{(0)}$.
\end{theorem} 
This improves on the above bound $O((n-m)^2 \bar\Le_A^2)$ for most values of the parameters (recall that $\kappa_A\le\bar\kappa_A^2$). Moreover, this bounds the running time for steepest descent for an arbitrary cost function $c$, not necessarily of the form \eqref{eq:c-u}.

Both these bounds are independent of $b$, however, $\kappa_A$ and $\bar\kappa_A$ may be exponentially large in the encoding length $L_A$ of the matrix $A$. In contrast, Theorem~\ref{thm:DLHL}\ref{i:deepest} yields a polynomial bound $O(n L_{A,b})$ on the number of deepest-descent iterations, where $L_{A,b}$ is the encoding length of $(A,b)$. In what follows, we review a new circuit augmentation algorithm from \cite{dadush2021circuit} that achieves a $\log\kappa_A$ dependence; the running time is bounded as $O(n^3 L_A)$, independently from $b$.

\subsection{A circuit augmentation algorithm with $\log\kappa_A$ dependence}\label{sec:log-kappa}
Recall that the diameter bound Theorem~\ref{thm:diam-augment} is non-algorithmic in the sense that the augmentation steps rely on knowing the optimal solution $x^*$.
Dadush et al.~\cite{dadush2021circuit} complemented this with an efficient circuit augmentation algorithm, assuming oracles are provided for certain circuit directions.

\begin{theorem}[\cite{dadush2021circuit}]\label{thm:opt-augment}
    Consider the primal of \ref{LP_primal_dual}. Given a feasible solution, there exists a circuit augmentation algorithm that finds an optimal solution or concludes unboundedness using $O(n^3\log(n+\kappa_A))$ circuit augmentations.
\end{theorem}
	The main circuit augmentation direction used in the paper for optimization is a step defined as \textsc{Ratio-Circuit}, a generalisation of the previously mentioned augmentation step by Wallacher \cite{Wallacher} for minimum cost flows.
It finds a circuit that is a basic optimal solution to the following linear system:
\begin{equation}
    \label{sys:minratio-i}
    \min \; \pr{c}{z}\, \quad \mathrm{s.t.}\quad
    Az =0\, , \, 
    \pr{w}{z^-}  \leq 1\, .
\end{equation}
Equivalently, the goal is to minimize the cost to weight ratio of a circuit, where the unit weight of increasing a variable $x_i$ is $0$ and decreasing it is $w_i$. This can be seen as an efficiently implementable relaxation of the  deepest-descent direction: for suitable weights, it achieves a geometric decrease in the objective value. For a vector $x\in\R^n_+$, we let $1/x\in(\R\cup\{\infty\})^n$ denote the vector $w$ with $w_i=1/x_i$ (in particular, $w_i=\infty$ if $x_i=0$).

\begin{lemma}[{\cite{mccormick-shioura-not-strongly}}]
    Let $\OPT$ denote the optimum value of \ref{LP_primal_dual}.
    Given a feasible solution $x$ to \ref{LP_primal_dual}, let 
        $g$ be the elementary vector returned by {\sc Ratio-Circuit}$(A,c,1/x)$, and $x'$ the next iterate. Then,
    \[\pr{c}{x'}-\OPT\le \left(1-\frac{1}{n}\right)\left(\pr{c}{x}-\OPT\right)\, .\]
\end{lemma}
\begin{proof}
Let $x^*$ be an optimal solution to \ref{LP_primal_dual}, and let $z=(x^*-x)/n$. Then, $z$ is feasible to 
\eqref{sys:minratio-i} for $w=1/x$. The claim easily follows by noting that $\pr{c}{g}\le \pr{c}{z}=(\OPT-\pr{c}{x})/n$, and noting that $x+g$ is feasible since  $\pr{1/x}{g^-}\le 1$.
\end{proof}

Repeated application of a \textsc{Ratio-Circuit} step thus provide an iterate whose optimality gap decreases by a constant factor within $O(n)$ iterations. However, as noted in \cite{mccormick-shioura-not-strongly}, using only this rule does not even finitely terminate already for minimum cost flows.

\paragraph{Support Circuits}
For this reason,
\cite{dadush2021circuit} occasionally uses a second
circuit augmentation step called {\sc Support-Circuit}. Roughly speaking, when given a non-basic feasible point in the system \ref{LP_primal_dual}, one can efficiently augment around a circuit $g$ such that $\pr{c}{g} \le 0$ and thereby reduce its support while maintaining the objective. 

On a high level, the need for such an operation becomes clear when considering following example. Assume $c \equiv 1$ and further assume that we are given an iterate $x$ with $\|x\| \gg \kappa_A \|y\|$ for some basic solution $y$. Then geometric progress in the objective $\pr{c}{x}$ can be achieved by just reducing the norm of $x$, but just geometric progress would not give the desired bound in the number of circuit augmentations. Note that the norm of all basic solutions lies within a factor of $\operatorname{poly}(n) \kappa_A\|y\|$ by Proposition~\ref{prop:kappa-max}. Therefore, it is helpful to reduce the support through at most $n$ \textsc{Support-Circuit} operations until a basic solution is reached instead of applying \textsc{Ratio-Circuit}. Subsequent applications of \textsc{Ratio-Circuit} will now, again due to Proposition~\ref{prop:kappa-max}, not be able to reduce the norm of $x$ by more than a factor of $\operatorname{poly}(n)\kappa_A$, a fact that will be exploited in the proximity arguments. 

\medskip
The main progress in the algorithm in Theorem~\ref{thm:opt-augment} is identifying a new index $i$  such that  $x_i=0$ in the current solution and $x_i^* = 0$ in any optimal solution $x^*$. Such a conclusion derives using a variant of the proximity theorem Theorem~\ref{thm:primal-fixing-upper}.
To implement the main subroutine that fixes a variable to $x_i=0$, a careful combination of \textsc{Ratio-Circuit}  and \textsc{Support-Circuit}  iterations is used.
Interestingly, the \textsc{Ratio-Circuit} iterations do not use the original cost function $c$, but a perturbed objective function $c'$. The main progress in the subroutine is identifying new `large' variables, similarly to the proof of Theorem~\ref{thm:diam-augment}. Perturbations are performed whenever a new large variable $x_j$ is identified.

\subsection{An improved bound for steepest-descent augmentation}
\label{sec:steepest-improved}
We now prove Theorem~\ref{thm:steepest-main}.
The proof follows the same lines as that of the Goldberg--Tarjan algorithm; see also \cite[Section 10.5]{amo} for the analysis. A factor $\log n$ improvement over the original bound was given in \cite{Radzik-Goldberg}. A key property in the original analysis is that for a flow around a cycle (i.e., an elementary vector), every edge carries at least $1/|V|$ fraction of the $\ell_1$-norm of the flow. This can be naturally replaced by the argument that for every elementary flow $g$, the minimum nonzero value of $|g_i|$ is at least $\|g\|_1/(1+(m-1)\kappa_A)$. 

The Goldberg--Tarjan algorithm has been generalized to separable convex minimization with linear constraints by Karzanov and McCormick \cite{Karzanov97}. Instead of $\Le_W$, they use the maximum entry in a Graver basis (see Section~\ref{sec:graver} below). Lemma 10.1 in their paper proves a weakly polynomial bound similar to Lemma~\ref{lem:weakly-poly} for the separable convex setting. However, no strongly polynomial analysis is given (which is in general not possible for the nonlinear setting).

\medskip

Our arguments will be based on the dual of \ref{LP-augment}:
\begin{equation}
\label{LP-augment-dual}
  \begin{aligned}
  \max \;  \pr{y}{b}-&\pr{u}{t} \\
  A^\top y + s-t &= c \\
  s,t & \geq 0\,. \\
  \end{aligned}
\end{equation}
Recall the primal-dual slackness conditions from Section~\ref{sec:hoffman}: if  $x$ is feasible to \ref{LP-augment} and $y\in\R^m$, they  are primal and dual optimal solutions if and only if $\pr{a_i}{y}\le c_i$ if $x_i<u_i$ and $\pr{a_i}{y}\ge c_i$ if $x_i>0$.

Let us start by formulating the steepest-descent direction as an LP. 
Let ${\bar A}=(A|-A)\in \R^{m\times (2n)}$ and $\bar c=\binom{c}{-c}\in \R^{2n}$. Clearly, $\Le_{\bar A}=\Le_A$.
For a feasible solution $x=x^{(t)}$ to \ref{LP-augment}, we define 
\emph{residual variable set} 
\[
N(x)=\{i\in [n]: x_i<u_i\}\cup \{n+j: j\in [n]: x_j>0\}\subseteq [2n]\, ,
\]
and consider  the system
\begin{equation}
\label{eq:minratio-primal}
\begin{aligned}
\min \; & \pr{\bar c}{z} \\ 
\bar A z&=0 \\
\pr{\1_{2n}}{z}&=1\\
z_{[2n] \setminus N(x)} &= 0\\
z&\ge 0\,.\\
\end{aligned}
\end{equation}
We can map a solution $z\in \R^{2n}$ to $g\in \R^n$ by setting $g_i=z_i-z_{n+i}$.
We will assume that $z$ is chosen as a basic optimal solution.
Observe that every basic feasible solution to this program maps to an elementary vector in $\ker(\bar A)$.
The dual program can be equivalently written as 
\begin{equation}
\label{eq:minratio-dual}
\begin{aligned}
\min \; & \varepsilon \\ 
\pr{\bar a_i}{y}&\le \bar c_i +\varepsilon \quad\forall i\in N(x)\,.
\end{aligned}
\end{equation}

For the solution $x$, we let $\varepsilon(x)$ denote the optimal solution to this dual problem; thus, the optimal solution to the primal is $-\varepsilon(x)$. 
If $\varepsilon=0$, then $x$ and $y$ are complementary primal and dual optimal solutions to \ref{LP-augment}. We first show that this quantity is monotone (a key step also in the analysis in \cite{DHL15}).

\begin{lemma}\label{lem:epsilon-mon} At every iteration of the circuit augmentation algorithm, $\varepsilon(x^{(t+1)})\le \varepsilon(x^{(t)})$.
\end{lemma}
\begin{proof}
Let $\varepsilon=\varepsilon(x^{(t)})$ and let $y$ be an optimal solution to \eqref{eq:minratio-dual} for $N(x^{(t)})$. We show that the same $y$ is also feasible for $N(x^{(t+1)})$; the claim follows immediately.  There is nothing to prove if $N(x^{(t+1)})\subseteq N(x^{(t)})$, so let $i\in N(x^{(t+1)})\setminus  N(x^{(t)})$.

Assume first $i\in [n+1,2n]$; let $i=n+j$. This means that $x^{(t)}_j=0<x_j^{(t+1)}$; therefore,  the augmenting direction $g$ has $g_{j}>0$. Thus, for the optimal solution $z$ to  \eqref{eq:minratio-primal}, we must have $z_j>0$. By primal-dual slackness, $\pr{a_j}{y}=c_j+\varepsilon$; thus, 
\[
\pr{\bar a_i}{y}=\pr{-a_j}{y}=-c_j-\varepsilon<-c_j=\bar c_j\, .
\]
The case $i\in [n]$ is analogous.
\end{proof}

The next lemma shows that within every $n$ iterations, $\varepsilon(x^{(t)})$ decreases by a factor depending on $\kappa_A$.
\begin{lemma}\label{lem:weakly-poly} For every iteration $t$, $\varepsilon(x^{(t+n)})\le \left(1-\frac{1}{1+(m-1)\kappa_A}\right) \varepsilon(x^{(t)})$.
\end{lemma}
\begin{proof}
Let us set $N=N(x^{(t)})$, $\varepsilon=\varepsilon(x^{(t)})$, and let $y=y^{(t)}$ be an optimal dual solution to \eqref{eq:minratio-dual} for $x^{(t)}$.
 Let 
 \[
 T := \{i \in N: \pr{\bar a_i}{y} > \bar c_i \} \subseteq [2n]\, ;
 \]
  that is, if $i\in [n]$ then $\pr{a_i}{y}>c_i$, and if $i\in [n+1,2n]$, $i=n+j$, then  $\pr{a_j}{y}<c_j$.  In particular, $|T\cap\{i,i+n\}|\le 1$ for every $i\in [n]$.
 Let $z^{(t)}$ be the basic optimal solution to \eqref{eq:minratio-primal} for $x^{(t)}$.
By complementary slackness, every $i\in \supp(z^{(t)})$ must have   $\pr{\bar a_i}{y}=\bar c_i+\varepsilon$, and thus,
 $\supp(z^{(t)})\subseteq T$.

\begin{claim}
Let us pick $k>t$ as the first iteration when for the basic optimal solution  $z^{(k)}$ to \eqref{eq:minratio-primal}, we have $\supp(z^{(k)})\setminus T\neq \emptyset$.
Then $k\le t+n$, and 
the solution $(y,\varepsilon)$ is still feasible for \eqref{eq:minratio-dual} for $x^{(k)}$.
\end{claim}
\begin{proof}
For $r\in[t,k-1]$, 
let  $T^{(r)}=T\cap N(x^{(r)})$. We show that $T^{(r+1)}\subsetneq T^{(r)}$. Since $|T|\le n$, this implies $k\le t+n$.
Let $z^{(r)}$ be the basic optimal solution  for \eqref{eq:minratio-primal}; recall that the augmenting direction is computed with $g_j=z^{(r)}_j-z^{(r)}_{n+j}$. By the choice of $k$, $\supp(z^{(r)})\subseteq T^{(r)}$. Thus, 
 we may only increase $x_i$ for $i\in T\cap [n]$ and decrease it for $i=j-n$ for $j\in T\cap [n+1,2n]$. Consequently, every index $i$ entering $N(x^{(r+1)})$ has $\pr{\bar a_i}{y}<\bar c_i$, and therefore $(y,\varepsilon)$ remains feasible throughout.

We now turn to the proof of $T^{(r+1)}\subsetneq T^{(r)}$.
Since we use a maximal augmentation, at least one index leaves $T^{(r)}$ at each iteration. We claim that $T^{(r+1)}\setminus T^{(r)}=\emptyset$. For a contradiction, assume  there exists $i\in T^{(r+1)}\setminus T^{(r)}$. If $i\in [n]$, then
$i+n$ must be in the support of $z^{(r)}$; in particular, $i+n\in T^{(r)}$. But this would mean that $\{i,i+n\}\subseteq T$, in contradiction with the definition of $T$. Similarly, for $i\in [n+1,2n]$.
\end{proof}

Let us now consider the optimal solution  $z=z^{(k)}$ to \eqref{eq:minratio-primal} at iteration $k$; by the above claim, $(y,\varepsilon)$ is still a feasible dual solution. Select an index  $j\in \supp(z)\setminus T$.
\[
\pr{-\bar c}{z}=\pr{\bar A^\top y - \bar c} z \le \varepsilon \sum_{i\in \supp(z)\setminus \{j\}} z_i = (1-z_j)  \varepsilon \le 
\left(1-\frac{1}{1+(m-1)\kappa_A}\right) \varepsilon\, .
\]
In the first inequality, we use that $\pr{\bar a_i}{y} - \bar c_i \le \varepsilon$ by the feasibility of $(y,\varepsilon)$, and $\pr{\bar a_j}{y} - \bar c_j \le 0$ by the choice of $j\notin T$. In the second equality, we use 
the constraint $\sum_i z_i=1$. The final inequality uses that
 $z$ is a basic solution, and therefore, an elementary vector in $\ker(\bar A)$. In particular $|\supp(z)|\le m$, and $z_i\le \kappa_{\bar A} z_j= \kappa_{A} z_j$. Consequently, $z_j\ge 1/(1+(m-1)\kappa_A)$.
\end{proof}

We say that the variable $j\in [2n]$ is \emph{frozen} at iteration $t$, if $j\notin N(x^{(t')})$ for any $t'\ge t$. Thus, for $j\in [n]$, $x_j=u_j$, and for $j\in [n+1,2n]$, $j=i+n$, $x_i=0$ for all subsequent iterations. We show that a new frozen variable can be found in every $O(nm\kappa_A\log(m\kappa_A))$ iterations; this implies
Theorem~\ref{thm:steepest-main}.

\begin{lemma}
For every iteration $t\ge 1$, there is a variable $j\in  N(x^{(t)})$ that is frozen at iteration $k$ for $k=t+O(nm\kappa_A\log(\kappa_A+n))$.
\end{lemma}
\begin{proof}
Let $\varepsilon=\varepsilon(x^{(t)})$. By Lemma~\ref{lem:weakly-poly}, we can choose $k=t+O(nm\kappa_A\log(n+\kappa_A))$ such that $\varepsilon'=\varepsilon(x^{(k)})<\varepsilon/(2n(\kappa_A+1))$.
Consider the primal and dual optimal solutions $(z,y,\varepsilon)$ to \eqref{eq:minratio-primal} and \eqref{eq:minratio-dual}  at iteration $t$ and $(z',y',\varepsilon')$ at iteration $k$. 

\begin{claim}\label{cl:largeindex} There exists a $j\in \supp(z)$ such that $\pr{\bar a_j}{y'}>\bar c_j + 2n(\kappa_A+1)\varepsilon'$. 
\end{claim}
\begin{proof}
For a contradiction, assume that $\pr{\bar a_j}{y'}-\bar c_j\le 2n(\kappa_A+1)\varepsilon'$ for every $j\in \supp(z)$.
 Then,
\[
\varepsilon= \pr{\bar c}{z}=\pr{\bar A^\top y -\bar c}{z} \le 2n(\kappa_A+1)\varepsilon'\sum_j z_j= 2n(\kappa_A+1)\varepsilon'\, , 
\]
contradicting the choice of $\varepsilon'$.
\end{proof}
We now show that all such indices are frozen at iteration $k$ by making use of Theorem~\ref{thm:primal-fixing-upper} on proximity. 
Let $x'=x^{(k)}$ and $x''=x^{(k'')}$ for any $k''>k$; let $(y'',\varepsilon'')$ be optimal to  \eqref{eq:minratio-dual} at iteration $k''$; we have $\varepsilon''\le\varepsilon'$ by Lemma~\ref{lem:epsilon-mon}.

Let us define the cost $c'\in\R^{n}$ by 
\[
c'_i:=\begin{cases}
\pr{\bar a_i}{y'}&\mbox{ if }0<x'_i<u_i\, \\
\max\{c_i, \pr{\bar a_i}{y'}\}&\mbox{ if }x'_i=u_i\, \\ 
\min\{c_i, \pr{\bar a_i}{y'}\}&\mbox{ if }x'_i=0\,.
\end{cases}
\]
If we replace the cost $c$ by $c'$, then  $x'$ and $y'$ satisfy complementary slackness, and hence are optimal to \ref{LP-augment} and \eqref{LP-augment-dual}. Moreover, the optimality of $(y',\varepsilon')$ to \eqref{eq:minratio-dual} guarantees that $\|c'-c\|_\infty\le \varepsilon'$.

We similarly construct $c''$ for $y''$, and note that $x''$ and $y''$ are primal and dual optimal solutions for the costs $c''$, $\|c''-c\|_\infty\le \varepsilon''$.  Further,
\[
\|c'-c''\|_1\le n\|c'-c''\|_\infty\le n\left(\|c'-c\|_\infty+\|c''-c\|_\infty\right)\le n(\varepsilon'+\varepsilon'')\le 2n\varepsilon'
\]
We thus apply Theorem~\ref{thm:primal-fixing-upper} for $(x',y')$ for $c'$ and $(x'',y'')$ for $c''$, showing that every variable $j$ as in Claim~\ref{cl:largeindex} must be frozen.
\end{proof}

\section{Circuits, integer proximity, and Graver bases}\label{sec:graver}
We now briefly discuss implications of circuit imbalances to the integer program (IP) of the form
\begin{equation}
\label{eq:IP} \tag{IP}
\begin{aligned}
\min \; &\pr{c}{x} \quad \\
Ax& =b \\
x &\geq 0, \\
x &\in \Z^n.
\end{aligned}
\end{equation}

Many algorithms for \eqref{eq:IP} solve the LP-relaxation first and deduce from the optimal solution of the relaxation information about the IP itself. The following proximity lemma shows that in case that \eqref{eq:IP} is feasible, the distance of an optimal integral solution to the optimal solution of the relaxation can be bounded in terms of max-circuit imbalance $\bar\kappa_A$. So, a local search within a radius of this guaranteed proximity will provide the optimal solution for the IP; see \cite[Proposition 4.1]{Lee89}.
\begin{lemma}\label{opt:IP}
Let $x^*$ be an optimal solution to \ref{LP_primal_dual}. Then there exists an optimal solution $\hat x$ to \eqref{eq:IP} such that $\|\hat x - x\|_1 \le n \bar\kappa_W$.
\end{lemma}
\begin{proof}
  Let $\hat x$ be an optimal solution to \eqref{eq:IP} that minimizes $\|\hat x - x^*\|_1$ and consider $w = x^* - \hat x \in W$ and a conformal circuit decomposition $w = \sum_{i = 1}^k \lambda_i g^{C_i}$ for some $k \le n$ and circuits $C_1, \ldots, C_k$ and $\lambda_1, \ldots, \lambda_k \ge 0$. Then, $\pr{c}{g^{C_i}} \le 0$ for all $i \in [k]$ as otherwise $x^* - \lambda_i g^{C_i}$ would be a feasible solution to the primal of \ref{LP_primal_dual} with strictly better objective than $x^*$. Further, note that $\lambda_i \le 1$ for all $i \in [k]$ as otherwise $\hat x - g^{C_i}$ is a feasible solution to \eqref{eq:IP} that has an objective as least as good as $\hat x$ and would in $\ell_1$ norm be strictly closer to $x^*$ than $\hat x$.   
  Therefore, 
  \[\|\hat x - x^*\|_\infty \leq \sum_{i = 1}^k \|g^{C_i}\|_\infty \le n\bar \kappa_W. \]
\end{proof}
Another popular and well-studied quantity in integer programming is the \emph{Graver basis}, defined as follows.
\begin{Def}[Graver basis]
The \emph{Graver basis} of a matrix $A$, denoted by $\mathcal{G}(A)$, consists of all $g \in \ker(A) \cap \Z^n$ such that there exists no $h \in (\ker(A) \cap \Z^n)\setminus \set{g}$ such that $g$ and $h$ are conformal and $|h_i| \le |g_i|$ for all $i \in [n]$. 
We can further define  
\begin{equation}
  \mathfrak g_1(A) := \max_{v \in \mathcal{G}(A)} \|v\|_1, \qquad
  \mathfrak g_\infty(A) := \max_{v \in \mathcal{G}(A)} \|v\|_\infty.
\end{equation}
\end{Def}
See \cite{bookDeLoera} for extensive treatment of the Graver basis and \cite{eisenbrand2019algorithmic} for more recent developments.  
Clearly, elementary vectors, scaled such that its entries have greatest common divisor equal to one belong to the Graver basis: $\set{g \in \EE(W) \cap \Z^n: \gcd(g) = 1} \subseteq \mathcal{G}(A)$. 

We will furthermore see how the max-circuit imbalance measure and Graver basis are related.

\begin{lemma} $\bar \kappa_A \le \mathfrak g_\infty(A) \le n\bar\kappa_A$.
\end{lemma}
\begin{proof}
The first inequality follows from the paragraph above, noting that 
  $$\set{g^C : C\in \circuits(W)} \subseteq \mathcal{G}(A),$$
for the normalized elementary vectors $g^C$ with $\lcm(g^C)=1$.
  For the second inequality, let $g \in \mathcal{G}(A)$ and $g = \sum_{i = 1}^k{\lambda_i g^{C_i}}$ be  a conformal circuit decomposition where $k \le n $. Note that $\lambda_i < 1$ for all $i \in [k]$ as otherwise $h^i$ would contradict that $g \in \mathcal{G}(A)$. Therefore, 
  \begin{equation}
    \|g\|_\infty \le \sum_{i = 1}^k \lambda_i \|g^{C_i}\|_\infty \le n \bar \kappa_A. 
  \end{equation} 
\end{proof}

Using the Steinitz lemma, Eisenbrand and Weismantel  \cite[Lemma 2]{eisenbrand2018blockstructure} gave a bound on $\mathfrak g_1(A)$ that only depends on $m$ but is independent of $n$:

\begin{theorem}
  Let $A \in \Z^{m \times n}$. Then $\mathfrak g_1(A) \le (2 m \|A\|_{\max} + 1)^m$.
\end{theorem}

\section{A decomposition conjecture}\label{sec:conjecture}

Let $W\subseteq \R^n$ be a linear space.
As the analogue of maximal augmentations, we say that a conformal circuit decomposition of $z\in W$ is \emph{maximal}, if it can be obtained as follows. If $z\in W$ is an elementary vector, return the decomposition containing the single vector $z$. Otherwise, select an arbitrary $g\in\EE(W)$ that is conformal with $z$ (in particular, $\supp(g)\subsetneq \supp(z)$), and set $g^1=\alpha g$ for the largest value $\alpha>0$ such that $z- g^1$ is conformal with $z$. Then, recursively apply this procedure to $z- g^1$ to obtain the other elementary vectors $g^2,\ldots,g^h$. We have $h\le n$, since the support decreases by at least one due to the maximal choice of $\alpha$.
If $\Le_W=\kd_W=1$, then it is easy to verify the following.

\begin{prop}\label{prop:kappa-1-decompose}
Let $W\subseteq \R^n$ be a linear space with $\Le_W=1$, and let $z\in W\cap \Z^n$. Then, for every maximal conformal circuit decomposition $z=\sum_{k=1}^h g^k$, we have $g^k\in \EE(W)\cap \Z^n$.
\end{prop}

We formulate a conjecture asserting that this property generalizes for arbitrary $\kd_W$ values. Note that in the conjecture, we only require the existence of some (not necessarily maximal) conformal circuit decomposition.
\begin{conj} \label{conj:decompose}
Let $W\subseteq\R^n$ be a rational linear subspace. Then, for every $z\in W\cap \Z^n$, there exists a conformal circuit decomposition
$z=\sum_{k=1}^h g^k$,  $h\le n$ such that each $g^k$ is  a $1/\dot\kappa_W$-integral  vector in $\EE(W)$.  
\end{conj}
Note that it is equivalent to require the same property for elements of the Graver basis $z\in \mathcal{G}(A)$. Hence, the conjecture asserts that every vector in the Graver basis is a `nice' combination of elementary vectors.

\medskip

We present some preliminary evidence towards this conjecture:
\begin{prop}
Let $W\subseteq\R^n$ be a rational linear subspace with $\Le^*_W=1$. Then, for every $z\in W\cap \Z^n$, and every maximal conformal circuit decomposition
$z=\sum_{k=1}^h g^k$, 
we have that $g^k$ is  a $1/\dot\kappa$-integral  vector in $\EE(W)$.  
\end{prop}
\begin{proof}
Assume $\Le_{DW}=1$ for some $D\in \bD_n$. By Theorem~\ref{thm:kappa-star-1}, 
we can select $D$ such that all diagonal entries $d_i=D_{ii}\in \Z$ and $d_i|\kd_W$. Let $z=\sum_{k=1}^h g^k$ be any maximal conformal circuit decomposition of $z\in W\cap \Z^n$. Clearly, $Dz=\sum_{k=1}^h Dg^k$ is also a
maximal conformal circuit decomposition of $Dz\in DW\cap \Z^n$. 
By Proposition~\ref{prop:kappa-1-decompose}, $Dg^k\in \EE(DW)\cap \Z^n$. Since $d_i \divides \kd_W$, this implies that $g^k$ is $1/\kd_W$-integral.
\end{proof}
By Theorem~\ref{thm:kappa-star}, this implies the conjecture whenever $\kd_W=p^\alpha$ for $p\in\primes$, $p>2$, $\alpha\in \N$. Let us now consider the case when $\kd_W$ is a power of 2. We verify the conjecture when the decomposition contains at most three terms.
\begin{prop}
Let $W\subseteq\R^n$ be a rational linear subspace with $\Le_W=2^\alpha$ for $\alpha\in \Z^n$. If $z\in W\cap \Z^n$ has a maximal conformal circuit decomposition
$z=\sum_{k=1}^h g^k$ with $h\le 3$, then each
 $g^k$ is  a $1/\dot\kappa$-integral  vector in $\EE(W)$.  
\end{prop}

\begin{proof}
Let us write the maximal conformal circuit decomposition in the form $z=\sum_{k=1}^h \lambda_k g^k$ such that $\lcm(g^k)=1$, and all entries $g^k_i\in\{\pm1,\pm 2,\pm 4,\ldots,\pm 2^\alpha\}$ for  $k\in [h]$, $i\in[n]$.
There is nothing to prove for $h=1$.
If $h=2$, then by the maximality of the decomposition, $\lambda_1 = \min_j\{{z_{j}}/{g^1_{j}}\}$. Hence, $\lambda_1$ is $1/2^\alpha$-integral. Consequently, both $\lambda_1 g^1$ and $\lambda_2 g^2=z-\lambda_1 g^1$ are  $1/2^\alpha$-integral.

If $h=3$, then
$\lambda_1 g^1$ is $1/2^\alpha$-integral as above. It also follows that  $\lambda_2 g^2$ and $\lambda_3 g^3$ are $1/2^\beta$-integral for some $\beta\ge \alpha$. Let us choose the smallest such $\beta$; we are done if $\beta=\alpha$.

Assume for a contradiction $\beta>\alpha$. Let $\mu_k=2^\beta \lambda_k$ for $k=1,2,3$. Thus, $\mu_k\in\Z$, $\mu_1$ is even, and at least one of $\mu_2$ and $\mu_3$ is odd. We show that both $\mu_2$ and $\mu_3$ must be odd. Let us first assume that $\mu_3$ is odd. There exists an $i\in [n]$ such that $|g^3_i|=1$. Then, $2^\beta z_i=\mu_1 g^1_i+\mu_2 g^2_i+\mu_3 g^3_i$ implies that $\mu_2$ must also be odd. Similarly, if $\mu_2$ is odd then $\mu_3$ must also be odd. 

Let us take any $j\in [n]$ such that $g^1_j=0$. Then, $2^\beta z_i=\mu_2 g^2_j+\mu_3 g^3_j$. Noting that $|g^2_j|$ and $|g^3_j|$ are powers of 2, both at most $2^\alpha$, it follows that $|g^2_j|=|g^3_j|$; by conformity, we have $g^2_j=g^3_j$. 

Consequently, $\supp(g^2-g^3)\subseteq \supp(g^1)$. Clearly, $g^2-g^3\in W\setminus\{ 0\}$, and the containment is strict by the maximality of the decomposition: there exists an index $i\in \supp(z)$ such that $z_j=\lambda_1 g^1_j$. This contradicts the fact that $g^1\in\EE(W)$.
\end{proof}

\section*{Acknowledgements}
The authors are grateful to Daniel Dadush for numerous inspiring discussions and joint work on circuit imbalances and linear programming, and to 
Luze Xu for pointing them to Jon Lee's papers \cite{Lee89,Lee90}.
The authors would also like to thank Jes{\'u}s De Loera, Martin Kouteck{\'{y}}, and the anonymous reviewers for their helpful comments and suggestions.
 
\clearpage
\bibliographystyle{customalpha}
\bibliography{circuits}

\appendix
\section{Appendix: Proof of Proposition~\ref{prop:counterexample}}

\counterex*

\begin{proof}
We know all other representations of the space like $\tilde{A}$ such that $\ker(\tilde{A}) = \ker(A)$ are of the form $\tilde{A} = BA$ where $B$ is a $2 \times 2$ invertible matrix. Since $A_{11} = 1$ then to get an integral $\tilde{A}$ we need to have integer $B_{11}$ and $B_{21}$. Furthermore since the g.c.d. of the numbers in the second column is equal to $1$, then $B_{12}$ and $B_{22}$ should be integers as well.

It can be verified by computer that the only $2\times 1$ matrices like $v$ such that all entries of $v^TA$ are divisors of $5850$ are
$$\pm \begin{bmatrix}9 \\ -4 \end{bmatrix},\pm \begin{bmatrix}10 \\ -3 \end{bmatrix},\pm \begin{bmatrix}13 \\-3\end{bmatrix},\pm \begin{bmatrix}0 \\1\end{bmatrix}$$

Checking all different $2 \times 2$ matrices we can get these matrices:

\begin{align*}
\begin{bmatrix}9 & -4\\ 10 & -3 \end{bmatrix}\begin{bmatrix} 1 & 3 & 4 & 3 \\ 0 & 13 & 9 & 10 \end{bmatrix}&=\begin{bmatrix}\textbf{9} & \textbf{-25} & 0 & -13\\ \textbf{10} & \textbf{-9} & 13 & 0 \end{bmatrix} \\
\begin{bmatrix}13 & -3\\ 10 & -3 \end{bmatrix}\begin{bmatrix} 1 & 3 & 4 & 3 \\ 0 & 13 & 9 & 10 \end{bmatrix}&=\begin{bmatrix}\textbf{13} & 0 & \textbf{25} & 9\\ \textbf{10} & -9 & \textbf{13} & 0 \end{bmatrix} \\
\begin{bmatrix}9 & -4\\13 & -3 \end{bmatrix}\begin{bmatrix} 1 & 3 & 4 & 3 \\ 0 & 13 & 9 & 10 \end{bmatrix}&=\begin{bmatrix}\textbf{9} & -25 & 0 & \textbf{-13}\\ \textbf{13} & 0 & 25 & \textbf{9} \end{bmatrix} \\
\begin{bmatrix}0 & 1\\ 9 & -4 \end{bmatrix}\begin{bmatrix} 1 & 3 & 4 & 3 \\ 0 & 13 & 9 & 10 \end{bmatrix}&=\begin{bmatrix}0 & \textbf{13} & 9 & \textbf{10}\\ 9 & \textbf{-25} & 0 & \textbf{-13} \end{bmatrix} \\
\begin{bmatrix}0 & 1\\ 10 & -3 \end{bmatrix}\begin{bmatrix} 1 & 3 & 4 & 3 \\ 0 & 13 & 9 & 10 \end{bmatrix}&=\begin{bmatrix}0 & \textbf{13} & \textbf{9} & 10\\ 10 & \textbf{-9} & \textbf{13} & 0 \end{bmatrix} \\
\begin{bmatrix}0 & 1\\ 13 & -3 \end{bmatrix}\begin{bmatrix} 1 & 3 & 4 & 3 \\ 0 & 13 & 9 & 10 \end{bmatrix}&=\begin{bmatrix}0 & 13 & \textbf{9} & \textbf{10}\\ 13 & 0 & \textbf{25} & \textbf{9} \end{bmatrix}
\end{align*}
All of these matrices contain a $2 \times 2$ submatrix such that its inverse is not $\frac{1}{5850}$-integral.
\end{proof} 

\end{document}